\documentclass[11pt,twoside,reqno]{amsart}

\usepackage{microtype}
\usepackage[OT1]{fontenc}
\usepackage{type1cm}
\usepackage{amssymb}
\usepackage{mathrsfs}

\usepackage[left=2.3cm,top=3cm,right=2.3cm]{geometry}
\numberwithin{equation}{section}

\theoremstyle{plain}
\newtheorem{theorem}{Theorem}[section]
\newtheorem{maintheorem}{Theorem}

\newtheorem{proposition}[theorem]{Proposition}
\newtheorem{lemma}[theorem]{Lemma}

\theoremstyle{remark}
\newtheorem{remark}[theorem]{Remark}

\newtheorem{example}[theorem]{Example}
\newtheorem*{ack}{Acknowledgement}

\theoremstyle{definition}

\newcommand{\LL}{\mathcal{L}}

\newcommand{\MM}{\mathcal{M}}

\newcommand{\R}{\mathbb{R}}

\newcommand{\N}{\mathbb{N}}

\newcommand{\iii}{\mathtt{i}}
\newcommand{\jjj}{\mathtt{j}}
\newcommand{\kkk}{\mathtt{k}}

\newcommand{\eps}{\varepsilon}
\newcommand{\fii}{\varphi}

\newcommand{\ualpha}{\overline{\alpha}}
\newcommand{\lalpha}{\underline{\alpha}}

\newcommand{\la}{\langle}
\newcommand{\ra}{\rangle}

\DeclareMathOperator{\dimh}{dim_H}

\DeclareMathOperator{\dimaff}{dim_{aff}}

\DeclareMathOperator{\Aut}{Aut}
\DeclareMathOperator{\End}{End}
\DeclareMathOperator{\linspan}{span}

\DeclareMathOperator{\sgn}{sgn}

\begin{document}

\title[Structure of equilibrium states on self-affine sets]{Structure of equilibrium states on self-affine sets and strict monotonicity of affinity dimension}

\author{Antti K\"aenm\"aki}
\address{Department of Mathematics and Statistics \\
         P.O. Box 35 (MaD) \\
         FI-40014 University of Jyv\"askyl\"a \\
         Finland}
\email{antti.kaenmaki@jyu.fi}

\author{Ian D. Morris}
\address{Department of Mathematics\\
	University of Surrey\\
	Guildford GU2 7XH\\
	United Kingdom}
\email{i.morris@surrey.ac.uk}

\date{\today}

\begin{abstract}
A fundamental problem in the dimension theory of self-affine sets is the construction of high-dimensional measures which yield sharp lower bounds for the Hausdorff dimension of the set. A natural strategy for the construction of such high-dimensional measures is to investigate measures of maximal Lyapunov dimension; these measures can be alternatively interpreted as equilibrium states of the singular value function introduced by Falconer. Whilst the existence of these equilibrium states has been well-known for some years their structure has remained elusive, particularly in dimensions higher than two.  In this article we give a complete description of the equilibrium states of the singular value function in the three-dimensional case, showing in particular that all such equilibrium states must be fully supported. In higher dimensions we also give a new sufficient condition for the uniqueness of these equilibrium states.  As a corollary, giving a solution to a folklore open question in dimension three, we prove that for a typical self-affine set in $\R^3$, removing one of the affine maps which defines the set results in a strict reduction of the Hausdorff dimension.
\end{abstract}

\maketitle

\section{Introduction}

If $f_1,\ldots,f_N$ are contractions of a complete metric space $X$ it is well-known that there exists a unique nonempty compact set $E \subset X$ such that $E=\bigcup_{i=1}^N f_i(E)$; see \cite{Hutchinson1981}. In this circumstance the tuple $(f_1,\ldots,f_N)$ is called an \emph{iterated function system (IFS)} and $E$ its \emph{attractor}. Iterated function systems have been extensively studied since the 1980s as idealised models for the fractal structure of attractors and repellers of dynamical systems. A central problem in the study of iterated function systems is to calculate or estimate the dimension of the attractor $E$ for various notions of fractal dimension, most especially the Hausdorff dimension.

Particular interest has been given to the case of \emph{affine} and \emph{similitude} iterated function systems, where the ambient space $X$ is given by $\mathbb{R}^d$ for some $d \in \N$ and the contractions $f_i$ take the form $f_i(x) = A_ix+v_i$ for certain (usually invertible) linear maps $A_i$ and vectors $v_i$. Any iterated function system of this type is called \emph{affine}. In the special case where each of the linear maps $A_i$ is a scalar multiple of an isometry, the system is more usually called \emph{similitude}. The associated attractors are called \emph{self-affine} and \emph{self-similar}, respectively. The dimension theory of self-similar sets satisfying the open set condition -- a condition on the transformations $f_i$ which guarantees that the images $f_i(E)$ do not overlap significantly for different $i$ -- was completely resolved in 1981 by Hutchinson \cite{Hutchinson1981}. Since that time research attention has been divided between the study of self-similar sets and measures which do not satisfy the open set condition (see e.g.\ \cite{Hochman2014, Shmerkin2014}) and the dimension theory of self-affine sets and measures which are not self-similar. This article is concerned with the latter field of investigation, which has been a source of subtle and persistent open problems since it was first substantially investigated in the 1980s; see \cite{Bedford1984, Falconer1988, McMullen1984}. In many cases the problem can be made more tractable either by assuming some randomness in the defining affine IFS, as in \cite{BarralFeng2013, FengShmerkin2014, JordanPollicottSimon2007}, or by imposing special relations between the affine maps, as in \cite{Baranski2007, DasSimmons2016, Fraser2012}. Only very recently has the general deterministic case started to become accessible to researchers (see e.g.\ \cite{BaranyKaenmaki2015, MorrisShmerkin2016}), and then only for self-affine subsets of the plane.

A seminal early paper by Falconer \cite{Falconer1988} gave a formula for the Hausdorff dimension of a ``typical'' self-affine set. The \emph{affinity dimension} of a self-affine set is a number defined in terms of the linear parts of the affine transformations $f_i$ which was shown by Falconer to be equal to the Hausdorff dimension for Lebesgue-almost-every choice of the additive parts $v_i$, subject to the additional assumption that the affinities contract $\mathbb{R}^d$ by a factor stronger than $\frac{1}{2}$. In certain exceptional cases the Hausdorff dimension can be strictly smaller than the affinity dimension; for some examples of this phenomenon, see e.g.\ \cite{Bedford1984, DasSimmons2016, Fraser2012, McMullen1984}. Since Falconer's theorem was proved, a long-standing topic of investigation has been to find testable sufficient conditions for the Hausdorff dimension of a self-affine set to equal its affinity dimension, see e.g.\ \cite{Barany2015, HueterLalley1995, KaenmakiShmerkin2009, MorrisShmerkin2016}.
Falconer was able to show unconditionally in \cite{Falconer1988} that the affinity dimension of a self-affine set is always an upper bound for the Hausdorff dimension, and the challenge of this problem is therefore to bound the Hausdorff dimension from below. The standard approach to problems of this type is to construct measures on a given self-affine set whose Hausdorff dimension approximates the anticipated value (in this case affinity dimension) from below, and it is with this project that our research is ultimately concerned.

If an affine IFS $(f_1,\ldots,f_N)$ on $\mathbb{R}^d$ is given, then for every sequence $(i_n)_{n=1}^\infty \in \{1,\ldots,N\}^{\mathbb{N}}$ the limit
\begin{equation}\label{eq:projection}
  \lim_{n\to \infty} f_{i_1} \circ f_{i_2} \circ \cdots \circ f_{i_n}(x)
\end{equation}
exists for every $x \in \mathbb{R}^d$ and is constant with respect to the choice of initial point $x$. This observation gives rise to a natural projection $\pi \colon \{1,\ldots,N\}^{\mathbb{N}}\to \mathbb{R}^d$ associated to the iterated function system $(f_1,\ldots,f_N)$ which takes each sequence $(i_n)_{n=1}^\infty  \in \{1,\ldots,N\}^{\mathbb{N}}$ to the unique limit point of \eqref{eq:projection} in $\mathbb{R}^d$ which corresponds to that sequence. It is easy to see that the projection of the symbolic space $\{1,\ldots,N\}^\N$ is simply the entire self-affine set $E$. A natural approach to the construction of high-dimensional measures on the self-affine set is as projections of shift-invariant measures on the symbolic space. In the self-similar case it is sufficient to consider projections of Bernoulli measures, and this is the method used by Hutchinson \cite{Hutchinson1981} to resolve the self-similar case. In the self-affine case, the appropriate measures must be constructed via a variational principle and arise as equilibrium states of the singular value function (defined in the following section). While the existence of these equilibrium states has been known for some time (\cite{Kaenmaki2004}) their structure has remained poorly understood, and in dimensions higher than two it is not even known whether or not the number of ergodic equilibrium states associated to a self-affine IFS is finite.

In this article, we conduct the first detailed investigation of these equilibrium states in dimensions higher than two. We completely characterise the equilibrium states in dimension three and compute exactly the maximum possible number of ergodic equilibrium states in that dimension. We also give a new general sufficient condition for the existence of a unique ergodic equilibrium state in arbitrary dimensions, and when this condition holds the unique equilibrium state additionally enjoys a certain natural Gibbs property. As a concrete application of this work we show that for three-dimensional affine IFS defined by invertible affinities the removal of one of the affine transformations $f_i$ strictly reduces the affinity dimension of the associated self-affine set. In particular, this implies via the theorem of Falconer that for almost every three-dimensional self-affine set with contraction coefficient smaller than $\frac{1}{2}$, the Hausdorff dimension of the self-affine set is strictly reduced when one of the affine transformations is removed. This answers a folklore open problem in dimension three or lower which has recently been propagated by Schmeling. Examples are given to show that invertibility of the affinities is necessary for this property and that invertible exceptional cases exist.

\section{Statement of results} \label{se:sec2}

To state our results formally we first summarise some foundational results and definitions. Many of these preliminaries will, for the moment, be asserted without proof, with rigorous treatments being deferred to the following section. We recall that the \emph{singular values} $\alpha_1(A),\ldots,\alpha_d(A)$ of a real $d \times d$ matrix $A$ are defined to be the square roots of the non-negative real eigenvalues of the positive semidefinite matrix $A^TA$ listed in decreasing order with repetition according to multiplicity. If $0 \leq s \leq d$ then the \emph{singular value function} of $A$ with parameter $s$, denoted by $\fii^s(A)$, is defined to be the real number
\[
  \fii^s(A)=\alpha_1(A) \cdots \alpha_{\lfloor s \rfloor}(A) \alpha_{\lceil s \rceil}(A)^{s-\lfloor s \rfloor}.
\]
Intuitively, the value $\fii^s(A)$ represents a measurement of the $s$-dimensional volume of the image of the Euclidean unit ball under $A$. The function $(A,s)\mapsto \fii^s(A)$ is continuous in both $A \in GL_d(\mathbb{R})$ and $s\in [0,d]$, and satisfies $\fii^s(AB)\leq \fii^s(A)\fii^s(B)$.

If an $N$-tuple of $d \times d$ matrices $(A_1,\ldots,A_N)\in GL_d(\mathbb{R})^N$ is given (where here and throughout we assume $N \geq 2$) then for each $s \in [0,d]$ we define the \emph{singular value pressure} of $\mathsf{A}=(A_1,\ldots,A_N)$ with parameter $s$ to be the quantity
\[
  P_{\mathsf{A}}(\fii^s)=\lim_{n\to\infty} \tfrac{1}{n}\log \sum_{i_1,\ldots,i_n=1}^N \fii^s\left(A_{i_1}\cdots A_{i_n}\right)
\]
which exists by subadditivity. For a fixed invertible tuple $\mathsf{A}=(A_1,\ldots,A_N)$ the singular value pressure depends continuously on $s$, and when each $A_i$ is a contraction in the Euclidean norm it is a strictly decreasing function of $s$. In the latter case the \emph{affinity dimension} of $\mathsf{A}$ is defined to be the unique zero of $s \mapsto P_{\mathsf{A}}(\varphi^s)$ for $s \in[0,d]$ when such a zero exists, and $d$ otherwise, in which case $P_{\mathsf{A}}(\fii^s)>0$ for all $s \in [0,d]$. If $f_1,\ldots,f_N \colon \mathbb{R}^d \to \mathbb{R}^d$ are affine contractions defined by $f_i(x)=A_ix+v_i$ for all $x \in \mathbb{R}^d$ then by the classical result of Falconer \cite{Falconer1988} the affinity dimension of $\mathsf{A}$ is an upper bound for the Hausdorff dimension of the associated self-affine set, and as mentioned previously this upper bound is attained for Lebesgue-almost-every choice of the additive parts $v_i$ when the norm of each $A_i$ is less than one half.

The singular value pressure and affinity dimension of $(A_1,\ldots,A_N)$ are related in essential ways to certain properties of shift-invariant measures on the associated space of sequences $\{1,\ldots,N\}^{\mathbb{N}}$. Let us fix $\mathsf{A}=(A_1,\ldots,A_N) \in GL_d(\mathbb{R})^N$ and $s \in [0,d]$, where each $A_i$ is assumed to be a contraction. If $\mu$ is a Borel probability measure on the compact metrisable space $\{1,\ldots,N\}^{\mathbb{N}}$ which is ergodic and invariant with respect to the shift transformation $(i_n)_{n=1}^\infty \mapsto (i_{n+1})_{n=1}^\infty$, then we define
\[
  \lambda_{\mathsf{A}}(\fii^s,\mu)=\lim_{n\to\infty}\tfrac{1}{n}\int \log \fii^s(A_{i_1}\cdots A_{i_n})\,\mathrm{d}\mu\left[(i_n)_{n=1}^\infty\right]
\]
which is well-defined by subadditivity. The function
\[
  s \mapsto h(\mu)+\lambda_{\mathsf{A}}(\fii^s,\mu),
\]
where $h$ denotes Kolmogorov-Sinai entropy, is then also continuous and strictly decreasing and has at most one zero in $[0,d]$. We define the \emph{Lyapunov dimension} of $\mu$ to be this unique zero when it exists, and $d$ when it does not. The projection $\pi_*\mu$ of the measure $\mu$ onto the self-affine set $E\subset \mathbb{R}^d$ always has Hausdorff dimension bounded above by the Lyapunov dimension of $\mu$ and for fixed $\mu$ and $\mathsf{A}$, the Lyapunov dimension gives the exact value of the Hausdorff dimension of the projected measure $\pi_*\mu$ for Lebesgue-almost-every additive part; see \cite{Rossi2014}. 

For each $s \in [0,d]$ the singular value pressure may be characterised variationally as
\[
  P_{\mathsf{A}}(\fii^s)=\sup (h(\mu)+\lambda_{\mathsf{A}}(\fii^s,\mu))
\]
where the supremum is taken over all shift-invariant Borel probability measures $\mu$ on $\{1,\ldots,N\}^{\mathbb{N}}$. This supremum is always attained by at least one ergodic measure which we call a \emph{$\fii^s$-equilibrium state}; see \cite{Kaenmaki2004}. It was observed in \cite{KaenmakiVilppolainen2010} that such an equilibrium state is not necessarily unique. In general, the number of ergodic measures which can attain this supremum is unknown. This question was brought up already in \cite{Kaenmaki2003}. Importantly, if $s$ is equal to the affinity dimension of $\mathsf{A}$ then any $\fii^s$-equilibrium state is a measure of maximal Lyapunov dimension. The search for lower bounds on the Hausdorff dimension of self-affine sets thus leads naturally via the study of measures of maximal Lyapunov dimension to the investigation of equilibrium states of the singular value function.

In the parameter ranges $0 \leq s \leq 1$ and $d-1\leq s \leq d$, the singular value function simplifies respectively to $\fii^s(A)=\|A\|^s$ and $\fii^s(A)=|\det A|^{s-(d-1)} \|A^{\wedge(d-1)}\|^{d-s}$. Equilibrium states associated to the potential $\|\cdot\|^s$ as opposed to the potential $\fii^s$ have proven relatively easy to understand (see e.g.\ \cite{FengKaenmaki2011, Morris2016, MorrisShmerkin2016}) and in particular this allows the equilibrium states of the singular value function in two dimensions to be completely described, since when $d=2$ the set $[0,1]\cup [d-1,d]$ constitutes the entire parameter range of $s$. This has allowed considerable progress to be made in the dimension theory of planar self-affine sets by the method of showing that suitable measures on $\{1,\ldots,N\}^{\mathbb{N}}$ project to measures whose Hausdorff dimension matches their Lyapunov dimension; see e.g.\ \cite{BaranyRams2015, MorrisShmerkin2016}. In order to extend this strategy to higher-dimensional self-affine IFS, then, it is necessary that the equilibrium states of the singular value function in dimensions higher than two be understood.

Our first main result shows that in the three-dimensional case, there can be at most six distinct ergodic $\fii^s$-equilibrium states and that this number can be achieved. The proof of the result is given in \S \ref{sec:3d}.

\begin{maintheorem}\label{thm:3d}
  Let $0 < s < 3$ and $\mathsf{A}\in GL_3(\mathbb{R})^N$. Then the maximum possible number of distinct ergodic $\varphi^s$-equilibrium states of $\mathsf{A}$ is $6$, if $1<s<2$, and $3$, if otherwise, and every equilibrium state is fully supported.
\end{maintheorem}

As an application of this main result we are able to solve a folklore open question concerning a dimension drop on self-affine sets in dimension three. The question asks whether removing one of the defining affine maps results in a strict reduction of the Hausdorff dimension. For certain highly degenerate choices of the affine transformations $f_i$ it is possible to obtain counterexamples (for example, see Example \ref{ex:no-schmeling}), so the question is about the generic behaviour. During recent years the question has been propagated by Schmeling. There is some evidence that this result could be used to calculate the dimension of a solenoid; see Hasselblat and Schmeling \cite{HasselblattSchmeling2004}.

\begin{maintheorem} \label{thm:schmeling}
  Let $\mathsf{A} = (A_1,\ldots,A_N) \in GL_3(\R)^N$ be such that $P_{\mathsf{A}}(\fii^3) \le 0$ and $\|A_i\| < \tfrac{1}{2}$ for all $i \in \{ 1,\ldots,N \}$. If $E_{\mathsf{v}}' \subset E_{\mathsf{v}} \subset \R^3$ are nonempty compact sets satisfying
  \begin{equation*}
    E_{\mathsf{v}}' = \bigcup_{i=1}^{N-1} A_i(E_{\mathsf{v}}')+v_i \quad \text{and} \quad E_{\mathsf{v}} = \bigcup_{i=1}^N A_i(E_{\mathsf{v}})+v_i
  \end{equation*}
  for all $\mathsf{v} = (v_1,\ldots,v_N) \in (\R^3)^N$, then
  \begin{equation*}
    \dimh(E_{\mathsf{v}}') < \dimh(E_{\mathsf{v}})
  \end{equation*}
  for $\LL^{3N}$-almost all $\mathsf{v} \in \R^{3N}$.
\end{maintheorem}

Theorem \ref{thm:schmeling} is proved in \S \ref{sec:affinity} and the proof is actually a simple consequence of the aforementioned variational principle and the fact that all the $\fii^s$-equilibrium states are fully supported. The proof of Theorem \ref{thm:3d}, on the other hand, is more involved. This proof splits into three subcases all of which are proved by using different methods. Observe that, by Feng and K\"aenm\"aki \cite[Theorem 1.7]{FengKaenmaki2011}, we may assume that $1<s<2$. The following result is based on an investigation of the Zariski-closed semigroup generated by $A_1,\ldots,A_N$ and is proved in \S \ref{sec:alg}. For the definitions of irreducible and strongly irreducible matrix tuples the reader is referred to \S \ref{sec:irred}.

\begin{maintheorem}\label{thm:intro-alg}
  Let $k \in \{ 0,\ldots,d-1 \}$, $k<s<k+1$, and $\mathsf{A}\in GL_d(\mathbb{R})^N$. If $\mathsf{A}^{\wedge k}$ and $\mathsf{A}^{\wedge(k+1)}$ are both irreducible, and one of them is strongly irreducible, then there exists a unique $\varphi^s$-equilibrium state of $\mathsf{A}$ and it is fully supported.
\end{maintheorem}

The requirement that $\mathsf{A}^{\wedge k}$ or $\mathsf{A}^{\wedge(k+1)}$ be strongly irreducible cannot be substantially reduced; see Example \ref{ex:irred} below.  It is worthwhile to note that the $\fii^s$-equilibrium state in the above theorem satifies a certain Gibbs property; see Remark \ref{rem:gibbs}.

We will see that in the three-dimensional case, if $\mathsf{A}$ is strongly irreducible then Theorem \ref{thm:intro-alg} can be applied to guarantee the existence of a unique $\varphi^s$-equilibrium state for all $1<s<2$. In the case where $\mathsf{A}$ is irreducible but not strongly irreducible we show in Proposition \ref{thm:3dim-permutation} that $\mathsf{A}$ is a tuple of generalized permutation matrices in some basis.  A matrix $A \in GL_d(\R)$ is a \emph{generalised permutation matrix} if every row and every column of $A$ has exactly one nonzero entry. Note that $A$ is a generalised permutation matrix if and only if it permutes the coordinate axes of $\mathbb{R}^d$. Let $P_d(\R) \subset GL_d(\mathbb{R})$ be the group of generalised permutation matrices. For this kind of tuple the structure of $\fii^s$-equilibrium states is described by the following theorem.

\begin{maintheorem}\label{thm:intro-perm}
  Let $k \in \{ 0,\ldots,d-1 \}$, $k < s < k+1$, and $\mathsf{A} \in P_d(\mathbb{R})^N$. Then the maximum possible number of distinct ergodic $\varphi^s$-equilibrium states of $\mathsf{A}$ is $(d-k)\binom{d}{k}$ and every equilibrium state is fully supported.
\end{maintheorem}

We prove Theorem \ref{thm:intro-perm} in \S \ref{sec:permutation} and its proof is based on finding an appropriate higher-dimensional auxiliary matrix tuple to which we can apply the theorem of Feng and K\"aenm\"aki \cite[Theorem 1.7]{FengKaenmaki2011}. The remaining case to investigate in the proof of Theorem \ref{thm:3d} is the reducible matrix tuples. In this case, the matrices are block-upper triangular in some basis and hence, the following theorem together with Proposition \ref{pr:3d-final} settles the proof.

\begin{maintheorem}\label{thm:intro-block}
  Let $\mathsf{A}=(A_1,\ldots,A_N) \in GL_d(\mathbb{R})^N$ and $0 <s< d$. If there exist integers $d_1,\ldots,d_\ell$ and real matrices $A_i^{(j,k)}$ such that $\sum_{i=1}^\ell d_i=d$ and
  \[
    A_i =
    \begin{pmatrix}
      A_i^{(1,1)} & A_i^{(1,2)} & \cdots & A_i^{(1,\ell)} \\
      0           & A_i^{(2,2)} & \cdots & A_i^{(2,\ell)} \\
      \vdots      & \vdots      & \ddots & \vdots \\
      0           & 0           & \cdots & A_i^{(\ell,\ell)}
    \end{pmatrix}
  \]
  for all $i \in \{ 1,\ldots,N\}$, where each matrix $A_i^{(j,k)}$ has dimension $d_j\times d_k$, then the set of $\varphi^s$-equilibrium states of $\mathsf{A}$ is precisely the set of $\varphi^s$-equilibrium states of $\mathsf{A}' = (A_1',\ldots, A_N') \in GL_d(\mathbb{R})^N$ defined by
  \[
    A_i' = 
    \begin{pmatrix}
      A_i^{(1,1)} & 0           & \cdots & 0 \\
      0           & A_i^{(2,2)} & \cdots & 0 \\
      \vdots      & \vdots      & \ddots & \vdots \\
      0           & 0           & \cdots & A_i^{(\ell,\ell)}
    \end{pmatrix}
  \]
  for all $i \in \{ 1,\ldots,N\}$.
\end{maintheorem}

Theorem \ref{thm:intro-block} is proved in \S \ref{sec:block-upper} and it substantially generalises the theorem of Falconer and Miao \cite[Theorem 2.5]{FalconerMiao2007} which treated the upper triangular as opposed to block-upper triangular case. Although Theorems \ref{thm:intro-alg}--\ref{thm:intro-block} are stated in arbitrary dimension, together they completely describe $\fii^s$-equilibrium states only in dimension three. Some obstacles to our understanding of the higher-dimensional case are discussed in detail in  \S \ref{sec:remarks-high} below.

The remainder of the article is structured as follows. In the following section we introduce some preliminary facts and lemmas common to the proofs of Theorems \ref{thm:3d} to \ref{thm:intro-block}. We then proceed to prove Theorems \ref{thm:intro-alg}, \ref{thm:intro-perm}, and \ref{thm:intro-block} before combining these results in the proof of Theorem \ref{thm:3d}. In the penultimate section of this article we prove Theorem \ref{thm:schmeling}. In the final section, we present examples to show that the strong irreducibility criterion of Theorem \ref{thm:intro-alg} cannot be removed, that the conclusion of Theorem \ref{thm:schmeling} can fail for certain degenerate choices of vector $\mathsf{v}$, and that that theorem can also fail when the affinities $f_i$ are allowed to be non-invertible.


\section{Preliminaries}

\subsection{Set of infinite words}
Fix $N \in \N$ such that $N \ge 2$ and equip the set of all infinite words $\Sigma = \{ 1,\ldots,N \}^\N$ with the usual ultrametric: the distance between two different words is defined to be $2^{-n}$, where $n$ is the first place at which the words differ. It is straightforward to see that $\Sigma$ is compact. The \emph{left shift} is a continuous map $\sigma \colon \Sigma \to \Sigma$ defined by setting $\sigma(\iii) = i_2 i_3 \cdots$ for all $\iii = i_1 i_2 \cdots \in \Sigma$.

Let $\Sigma_*$ be the free monoid on $\{ 1,\ldots,N \}$. The concatenation of two words $\iii \in \Sigma_*$ and $\jjj \in \Sigma_* \cup \Sigma$ is denoted by $\iii\jjj \in \Sigma_* \cup \Sigma$. The set $\Sigma_*$ is the set of all finite words $\{ \varnothing \} \cup \bigcup_{n \in \N} \Sigma_n$, where $\Sigma_n = \{ 1,\ldots,N \}^n$ for all $n \in \N$ and $\varnothing$ satisfies $\varnothing\iii = \iii\varnothing = \iii$ for all $\iii \in \Sigma_*$. For notational convenience, we set $\Sigma_0 = \{ \varnothing \}$. The word $i_2 \cdots i_n \in \Sigma_{n-1}$ is denoted by $\sigma(\iii)$ for all $n \in \N$ and $\iii = i_1 \cdots i_n \in \Sigma_n$.

The length of $\iii \in \Sigma_* \cup \Sigma$ is denoted by $|\iii|$. If $\iii \in \Sigma_*$, then we set
$
  [\iii] = \{ \iii\jjj \in \Sigma : \jjj \in \Sigma \}
$
and call it a \emph{cylinder set}. If $\jjj \in \Sigma_* \cup \Sigma$ and $1 \le n < |\jjj|$, we define $\jjj|_n$ to be the unique word $\iii \in \Sigma_n$ for which $\jjj \in [\iii]$. If $\jjj \in \Sigma_*$ and $n \ge |\jjj|$, then $\jjj|_n = \jjj$.

\subsection{Multilinear algebra} \label{sec:multilinear}
We recall some basic facts about the exterior algebra.
Let $\{ e_1,\ldots,e_d \}$ be the standard orthonormal basis of $\R^d$ and define
\begin{equation*}
  \wedge^k \R^d = \linspan\{ e_{i_1} \wedge \cdots \wedge e_{i_k} : 1 \le i_1 < \cdots < i_k \le d \}
\end{equation*}
for all $k \in \{ 1,\ldots,d \}$ with the convention that $\wedge^0\R^d=\mathbb{R}$. Recall that the wedge product $\wedge \colon \wedge^k\R^d \times \wedge^j\R^d \to \wedge^{k+j}\R^d$ is an associative and bilinear operator, anticommutative on the elements of $\R^d$.
This means that
\begin{equation} \label{eq:basicfact2}
  v \wedge w = (-1)^{kj} w \wedge v
\end{equation}
for all $v \in \wedge^k\R^d$ and $w \in \wedge^j\R^d$.
If $v \in \wedge^k\R^d$ can be expressed as a wedge product of $k$ vectors of $\R^d$, then $v$ is said to be \emph{decomposable}. Observe that e.g.\ $e_1 \wedge e_2 + e_3 \wedge e_4 \in \wedge^2\R^4$ is not decomposable.
The \emph{Hodge star operator} $* \colon \wedge^k\R^d \to \wedge^{d-k}\R^d$ is defined to be the bijective linear map satisfying
\begin{equation*}
  *(e_{i_1} \wedge \cdots \wedge e_{i_k}) = \sgn(i_1,\ldots,i_d) e_{i_{k+1}} \wedge \cdots \wedge e_{i_d}
\end{equation*}
for all $1 \le i_1 < \cdots < i_k \le d$, where $1 \le i_{k+1} < \cdots < i_d \le d$ are such that $\{ i_{k+1},\ldots,i_d \} = \{ 1,\ldots,d \} \setminus \{ i_1,\ldots,i_k \}$, and $\sgn(i_1,\ldots,i_d)=1$ if $(i_1,\ldots,i_d)$ is an even permutation of $\{ 1,\ldots,d \}$ and $\sgn(i_1,\ldots,i_d)=-1$ otherwise.
It is straightforward to see that
\begin{equation} \label{eq:basicfact1}
  *(*v) = (-1)^{k(d-k)} v
\end{equation}
for all $v \in \wedge^k\R^d$.

The group of $d \times d$ invertible matrices of real numbers is denoted by $GL_d(\R)$. This space has a topology induced from $\R^{d^2}$. If $A \in GL_d(\R)$, we define an invertible linear map $A^{\wedge k} \colon \wedge^k\R^d \to \wedge^k\R^d$ by setting
\[
(A^{\wedge k})(e_{i_1} \wedge \cdots \wedge e_{i_k}) = Ae_{i_1} \wedge \cdots \wedge Ae_{i_k}
\]
and extending by linearity. Observe that $A^{\wedge k}$ can be represented by a $\binom{d}{k} \times \binom{d}{k}$ matrix whose entries are the $k \times k$ minors of $A$. Using this and standard properties of determinants, it may be shown that
\begin{equation} \label{eq:morphism}
  (AB)^{\wedge k} = (A^{\wedge k})(B^{\wedge k}),
\end{equation}
i.e.\ $A \mapsto A^{\wedge k}$ is a morphism between the corresponding multiplicative linear groups. Furthermore, if $\alpha_1(A) \ge \cdots \ge \alpha_d(A) > 0$ are the singular values of $A$, that is, the square roots of the eigenvalues of the positive definite matrix $A^TA$, where $A^T$ is the transpose of $A$, then the products $\alpha_{i_1}(A) \cdots \alpha_{i_k}(A)$ are the singular values of $A^{\wedge k}$, for each $1 \le i_1 < \cdots < i_k \le d$. 
Furthermore, it is straightforward to see that
\begin{equation} \label{eq:basicfact4}
  *(A^{\wedge k}w) = A^{\wedge(d-k)}(*w)
\end{equation} 
for all $w \in \wedge^k\R^d$.

The inner product on $\wedge^k\R^d$ is defined by setting
\begin{equation} \label{eq:inner}
  \langle v, w\rangle_k=*(v\wedge *w)
\end{equation}
for all $v,w \in \wedge^k\R^d$. Thus, by \eqref{eq:basicfact1} and \eqref{eq:basicfact2}, we have
\begin{equation} \label{eq:basicfact3}
\begin{split}
  *v \wedge *w &= \langle *v,w \rangle_{d-k} \, e_1 \wedge \cdots \wedge e_d = \langle w,*v \rangle_{d-k} \, e_1 \wedge \cdots \wedge e_d \\
  &= w \wedge *(*v) = w \wedge (-1)^{k(d-k)} v  = (-1)^{2k(d-k)}  v \wedge w = v \wedge w
\end{split}
\end{equation}
for all $v \in \wedge^k\R^d$ and $w \in \wedge^{d-k}\R^d$.
The norm is defined by setting $|v|_k = \langle v,v \rangle_k^{1/2}$ for all $v \in \wedge^k\R^d$. It follows that $|v_1 \wedge \cdots \wedge v_k|_k$ is the $k$-dimensional volume of the parallelepiped with the vectors $v_1,\ldots,v_k$ as sides. The operator norm of the induced linear mapping $A^{\wedge k}$ is
\begin{equation} \label{eq:operator_norm_for_wedge}
  \| A^{\wedge k} \|_k = \max\{ |A^{\wedge k}v|_k : |v|_k = 1 \} = \alpha_1(A) \cdots \alpha_k(A).
\end{equation}

\subsection{Irreducibility} \label{sec:irred}
Let $\mathcal{A}$ be a set of matrices in $GL_d(\R)$. We say that $\mathcal{A}$ is \emph{irreducible} if there is no proper nontrivial linear subspace $V$ of $\R^d$ such that $A(V) \subset V$ for all $A \in \mathcal{A}$; otherwise $\mathcal{A}$ is called \emph{reducible}. The set $\mathcal{A}$ is \emph{strongly irreducible} if there does not exist a set $F$ which is equal to a finite union of proper nontrivial linear subspaces of $\R^d$ and satisfies $A(F) \subset F$ for all $A \in \mathcal{A}$. Furthermore, a tuple $\mathsf{A} = (A_1,\ldots,A_N) \in GL_d(\R)^N$ is \emph{irreducible} (resp.\ \emph{strongly irreducible}) if the corresponding set $\{ A_1,\ldots,A_N \}$ is irreducible (resp.\ \emph{strongly irreducible}). If $\mathsf{A}^{\wedge k} = (A^{\wedge k}_1, \ldots, A^{\wedge k}_N)$ is irreducible (resp.\ strongly irreducible) for some $k \in \{ 0,\ldots,d \}$, then we say that $\mathsf{A}$ is \emph{$k$-irreducible} (resp.\ strongly $k$-irreducible). For each $n \in \N$ and $\iii = i_1 \cdots i_n \in \Sigma_n$ we write $A_\iii = A_{i_1} \cdots A_{i_n} \in GL_d(\R)$.

\begin{lemma} \label{thm:irr_equiv1}
  If $\mathsf{A} = (A_1,\ldots,A_N) \in GL_d(\R)^N$, then the following conditions are equivalent:
  \begin{enumerate}
    \item The tuple $\mathsf{A}$ is irreducible.
    \item For every $0\neq v,w \in \R^d$ there is $\iii \in \Sigma_*$ such that $\langle v,A_\iii w \rangle \ne 0$.
    \item For every $0 \ne w \in \R^d$ it holds that $\linspan(\{ A_\iii w : \iii \in \Sigma_* \}) = \R^d$.
    \item The set $\{ A_\iii : \iii \in \Sigma_* \}$ is irreducible.
  \end{enumerate}
\end{lemma}

\begin{proof}
  The proof is similar to that of \cite[Lemma 2.6]{FalconerSloan2009} and hence omitted.
\end{proof}

\begin{remark}
  For a tuple $\mathsf{A} = (A_1,\ldots,A_N) \in GL_2(\R)^N$ of invertible $2 \times 2$ matrices reducibility is equivalent to the property that the matrices $A_i$ can simultaneously be presented (in some coordinate system) as upper triangular matrices; see \cite[Remark 2.4(1)]{KaenmakiReeve2014}. In the higher dimensional case, by \cite[Proposition 1.4]{FengKaenmaki2011}, the reducible tuple $\mathsf{A}$ can be presented (in some coordinate system) as a tuple of block-upper triangular matrices.
\end{remark}

\begin{lemma} \label{thm:irr_equiv2}
  Let $k \in \{ 0,\ldots,d \}$ and $\mathsf{A} = (A_1,\ldots,A_N) \in GL_d(\R)^N$. Then $\mathsf{A}$ is $k$-irreducible if and only if $\mathsf{A}$ is $(d-k)$-irreducible.
\end{lemma}

\begin{proof}
  By symmetry, we only need to prove the ``only if'' part. Indeed, fix $0\neq v,w \in \wedge^{d-k}\R^d$ and notice that $*v,*w \in \wedge^k\R^d$. Recalling Lemma \ref{thm:irr_equiv1}, the irreducibility of $\mathsf{A}^{\wedge k}$ implies that there exists $\iii \in \Sigma_*$ such that $\langle *v,A^{\wedge k}_\iii(*w) \rangle_k \ne 0$. By \eqref{eq:inner}, \eqref{eq:basicfact4}, and \eqref{eq:basicfact3}, we have
  \begin{align*}
    \langle v,A^{\wedge(d-k)}_\iii w \rangle_{d-k} &= *(v \wedge *(A^{\wedge(d-k)}_\iii w)) = *(v \wedge (A^{\wedge k}_\iii(*w))) \\
    &= *(*v \wedge *(A^{\wedge k}_\iii(*w))) = \langle *v,A^{\wedge k}_\iii(*w) \rangle_k \ne 0,
  \end{align*}
  where the bijectivity of $A$ and of $*$ are required to show that $*v$ and $A^{\wedge k}_\iii(*w)$ are nonzero. A second application of Lemma \ref{thm:irr_equiv1} finishes the proof.
\end{proof}

\subsection{Singular value function} \label{sec:svf-pressure}

Let $k \in \{ 0,\ldots,d-1 \}$ and $k \le s < k+1$. We define the \emph{singular value function} to be
\begin{equation} \label{eq:def-svf}
  \fii^s(A) = \| A^{\wedge k} \|_{k}^{k+1-s} \; \| A^{\wedge(k+1)} \|_{k+1}^{s-k} = \alpha_1(A) \cdots \alpha_{k}(A) \alpha_{k+1}(A)^{s-k}
\end{equation}
for all $A \in GL_d(\R)$ with the convention that $\|A^{\wedge 0}\|_0 = 1$. Observe that \eqref{eq:morphism} and the submultiplicativity of the operator norm imply
\begin{equation} \label{eq:cylinder2}
\begin{split}
  \fii^s(AB) &= \| (AB)^{\wedge k} \|_{k}^{k+1-s} \; \| (AB)^{\wedge(k+1)} \|_{k+1}^{s-k} \\ &\le \| A^{\wedge k} \|_{k}^{k+1-s} \; \| B^{\wedge k} \|_{k}^{k+1-s} \; \| A^{\wedge(k+1)} \|_{k+1}^{s-k} \; \| B^{\wedge(k+1)} \|_{k+1}^{s-k} = \fii^s(A) \fii^s(B)
\end{split}
\end{equation}
for all $A,B \in GL_d(\R)$. When $s \ge d$, we set $\fii^s(A) = |\det(A)|^{s/d}$ for completeness.

If $U,V \in GL_d(\mathbb{R})$ are isometries, then $\varphi^s(UAV)\leq \varphi^s(U)\varphi^s(A)\varphi^s(V)=\varphi^s(A)$ for all $A \in GL_d(\mathbb{R})$, and by symmetry $\varphi^s(A)\leq \varphi^s(UAV)$ since $U^{-1}$ and $V^{-1}$ are isometries as well. In particular, $\varphi^s(A)=\varphi^s(UAV)$ whenever $U$ and $V$ are isometries. If $A \in GL_d(\mathbb{R})$ is a diagonal matrix then the singular values of $A$ are simply the absolute values of the diagonal entries, so clearly
\begin{equation}\label{eq:svd-formula}
  \varphi^s(A)=\max\biggl\{\biggl(\prod_{i=1}^{k} \|Au_i\|\biggr)\|Au_{k+1}\|^{s-k} : u_1,\ldots,u_{k+1} \in S^{d-1} \text{ are pairwise orthogonal} \biggr\}
\end{equation}
for all diagonal matrices $A \in GL_d(\mathbb{R})$. Here $S^{d-1}$ is the unit sphere of $\mathbb{R}^d$. Since, by the singular value decomposition, every $A \in GL_d(\mathbb{R})$ can be written in the form $A=UDV$ where $U, V$ are isometries and $D$ is diagonal, it follows that \eqref{eq:svd-formula} holds for all $A \in GL_d(\mathbb{R})$. 

Fix $\mathsf{A} = (A_1,\ldots,A_N) \in GL_d(\R)^N$. If we let $\ualpha = \max\{ \alpha_1(A_i) : i \in \{ 1,\ldots,N \} \}$ and $\lalpha = \min\{ \alpha_d(A_i) : i \in \{ 1,\ldots,N \} \} > 0$, then it follows that
\begin{equation*}
  \fii^s(A_\iii)\lalpha^{\delta |\iii|}
  \le \fii^{s+\delta}(A_\iii)
  \le \fii^s(A_\iii)\ualpha^{\delta |\iii|}
\end{equation*}
for all $\delta \ge 0$ and $\iii \in \Sigma_*$. Moreover, \eqref{eq:cylinder2} implies
\begin{equation} \label{eq:sums_triv1}
  \sum_{\iii \in \Sigma_{n+m}} \fii^s(A_\iii) \le \biggl( \sum_{\iii \in \Sigma_n} \fii^s(A_\iii) \biggr)\biggl( \sum_{\iii \in \Sigma_m} \fii^s(A_\iii) \biggr)
\end{equation}
for all $n,m \in \N$. We define
\begin{equation} \label{eq:topo_def}
  P_{\mathsf{A}}(\fii^s) = \lim_{n \to \infty} \tfrac{1}{n} \log\sum_{\iii \in \Sigma_n} \fii^s(A_\iii) = \inf_{n \in \N} \tfrac{1}{n} \log\sum_{\iii \in \Sigma_n} \fii^s(A_\iii)
\end{equation}
and call it the \emph{singular value pressure} of $\mathsf{A}$. The limit above exists and equals to the infimum by the standard theory of subadditive sequences. It is easy to see that, as a function of $s$, the singular value pressure is continuous, strictly decreasing, and convex between any two consecutive integers. Furthermore, since $P_{\mathsf{A}}(\fii^0)=\log N > 0$ and $\lim_{s \to \infty} P_{\mathsf{A}}(\fii^s) = -\infty$ there exists unique $s \ge 0$ for which $P_{\mathsf{A}}(\fii^s)=0$. The minimum of $d$ and this $s$ is called the \emph{affinity dimension} of $\mathsf{A}$ and is denoted by $\dimaff(\mathsf{A})$.

It follows e.g.\ from \cite[Corollary 8.6.2]{GolubVanLoan1983} that the singular value function $\fii^s(A)$ is continuous as a function of $A$. Recently, it has been observed that the singular value pressure is continuous also as a function of $\mathsf{A}$. The following result is proved by Feng and Shmerkin \cite[Theorem 1.2]{FengShmerkin2014}, and subsequently re-proved by Morris \cite[Theorem 2.2]{Morris2015}.

\begin{theorem} \label{le:cty}
  If $0<s<d$, then the function $\mathsf{A} \mapsto P_{\mathsf{A}}(\fii^s)$ defined on $GL_d(\R)^N$ is continuous.
\end{theorem}

Let $k \in \{ 0,\ldots,d-1 \}$ and $k < s < k+1$. We say that $\mathsf{A}$ is \emph{$s$-irreducible} if for every $v_k,w_k \in \wedge^k\R^d$ and $v_{k+1},w_{k+1} \in \wedge^{k+1}\R^d$ there is $\iii \in \Sigma_*$ such that
\begin{equation*}
  \langle v_k,A^{\wedge k}_\iii w_k \rangle_k \ne 0 \quad \text{and} \quad \langle v_{k+1},A^{\wedge(k+1)}_\iii w_{k+1} \rangle_{k+1} \ne 0.
\end{equation*}
Observe that, by Lemma \ref{thm:irr_equiv1}, if $\mathsf{A}$ is $s$-irreducible, then it is $k$-irreducible and $(k+1)$-irreducible.
We say that $\mathsf{A}$ is \emph{$\fii^s$-quasimultiplicative} if there exists a constant $c \ge 1$ and $K \in \N \cup \{ 0 \}$ so that for every $\iii,\jjj \in \Sigma_*$ there is $\kkk \in \bigcup_{k=0}^K \Sigma_k$ such that
\begin{equation} \label{eq:irr}
  \fii^s(A_\iii)\fii^s(A_\jjj) \le c\fii^s(A_{\iii\kkk\jjj}).
\end{equation}
The following lemma is similar to \cite[Proposition 2.1]{FalconerSloan2009}, and is also a modification of \cite[Proposition 2.8]{Feng2009}.

\begin{lemma} \label{thm:irreducibility}
  Let $k \in \{ 0,\ldots,d-1 \}$ and $k < t < k+1$. If $\mathsf{A} = (A_1,\ldots,A_N) \in GL_d(\R)^N$ is $t$-irreducible, then $\mathsf{A}$ is $\fii^s$-quasimultiplicative for all $k < s < k+1$.
\end{lemma}

\begin{proof}
  We assume, contrary to the claim, that there exists $k<s<k+1$ such that for every $K \in \N$ there are $\iii_K,\jjj_K \in \Sigma_*$ so that
  \begin{equation} \label{eq:irr_antithesis}
    \fii^s(A_{\iii_K\kkk\jjj_K}) < \fii^s(A_{\iii_K})\fii^s(A_{\jjj_K})/K
  \end{equation}
  for all $\kkk \in \Sigma_*$ with $|\kkk| \le K$. For each $K \in \N$ we choose $v_{K,k},w_{K,k} \in \wedge^k\R^d$ such that $|v_{K,k}|_k = |w_{K,k}|_k =1$ and
  \begin{align*}
    \| A^{\wedge k}_{\iii_K} \|_k &= \| (A^{\wedge k}_{\iii_K})^T \|_k = |(A^{\wedge k}_{\iii_K})^Tv_{K,k}|_k, \\
    \| A^{\wedge k}_{\jjj_K} \|_k &= |A^{\wedge k}_{\jjj_K}w_{K,k}|_k.
  \end{align*}
  Defining
  \begin{equation*}
    v'_{K,k} = \frac{(A^{\wedge k}_{\iii_K})^Tv_{K,k}}{\| A^{\wedge k}_{\iii_K} \|_k} \quad \text{and} \quad
    w'_{K,k} = \frac{A^{\wedge k}_{\jjj_K}w_{K,k}}{\| A^{\wedge k}_{\jjj_K} \|_k},
  \end{equation*}
  Cauchy-Schwarz inequality gives
  \begin{equation*}
    \langle v'_{K,k},A^{\wedge k}_\kkk w'_{K,k} \rangle_k = \frac{\langle (A^{\wedge k}_{\iii_K})^Tv_{K,k},A^{\wedge k}_{\kkk\jjj_K}w_{K,k} \rangle_k}{\| A^{\wedge k}_{\iii_K} \|_k \, \| A^{\wedge k}_{\jjj_K} \|_k}  \le \frac{\| A^{\wedge k}_{\iii_K\kkk\jjj_K} \|_k}{\| A^{\wedge k}_{\iii_K} \|_k \; \| A^{\wedge k}_{\jjj_K} \|_k}
  \end{equation*}
  for all $\kkk \in \Sigma_*$. We define $v'_{K,k+1}, w'_{K,k+1} \in \wedge^{k+1}\R^d$ in an analogous way. Hence, by recalling \eqref{eq:irr_antithesis} and the definition of the singular value function \eqref{eq:def-svf}, we get
  \begin{equation*}
    \langle v'_{K,k},A^{\wedge k}_\kkk w'_{K,k} \rangle_k^{k+1-s} \, \langle v'_{K,k+1},A^{\wedge(k+1)}_\kkk w'_{K,k+1} \rangle_{k+1}^{s-k} \le 1/K
  \end{equation*}
  for all $\kkk \in \Sigma_*$ with $|\kkk| \le K$. We may now choose a subsequence and elements $v_k,w_k \in \wedge^k\R^d$ and $v_{k+1},w_{k+1} \in \wedge^{k+1}\R^d$ with $|v_k|_k = |w_k|_k = 1 = |v_{k+1}|_{k+1} = |w_{k+1}|_{k+1}$ so that $v'_{K,k} \to v_k$, $w'_{K,k} \to w_k$, $v'_{K,k+1} \to v_{k+1}$, and $w'_{K,k+1} \to w_{k+1}$ when $K \to \infty$ along the subsequence. Therefore
  \begin{equation*}
    \langle v_k,A^{\wedge k}_\kkk w_k \rangle_k^{k+1-s} \, \langle v_{k+1},A^{\wedge(k+1)}_\kkk w_{k+1} \rangle_{k+1}^{s-k} = 0
  \end{equation*}
  for all $\kkk \in \Sigma_*$. This contradicts the hypothesis of $t$-irreducibility.
\end{proof}

\subsection{Equilibrium states} \label{sec:equilibrium}
We denote the collection of all Borel probability measures on $\Sigma$ by $\MM(\Sigma)$, and endow it with the weak$^*$ topology. We say that $\mu \in \MM(\Sigma)$ is \emph{fully supported} if $\mu([\iii]) > 0$ for all $\iii \in \Sigma_*$. Let
\[
  \MM_\sigma(\Sigma) = \{ \mu \in \MM(\Sigma) : \mu \text{ is $\sigma$-invariant} \},
\]
where \emph{$\sigma$-invariance} of $\mu$ means that
$\mu([\iii]) = \mu(\sigma^{-1}([\iii])) = \sum_{i=1}^N \mu([i\iii])$
for all $\iii \in \Sigma_*$. Observe that if $\mu \in \MM_\sigma(\Sigma)$, then $\mu(A)=\mu(\sigma^{-1}(A))$ for all Borel sets $A \subset \Sigma$.
We say that $\mu$ is \emph{ergodic} if $\mu(A) = 0$ or $\mu(A) = 1$ for every Borel set $A \subset \Sigma$ with $A = \sigma^{-1}(A)$. Recall that the set $\MM_\sigma(\Sigma)$ is compact and convex with ergodic measures as its extreme points.

If $\mu \in \MM_\sigma(\Sigma)$, then we define the \emph{entropy} $h$ of $\mu$ by setting
\begin{equation*}
  h(\mu) = \lim_{n \to \infty} \tfrac{1}{n} \sum_{\iii \in \Sigma_n} -\mu([\iii]) \log\mu([\iii]) = \inf_{n \in \N} \tfrac{1}{n} \sum_{\iii \in \Sigma_n} -\mu([\iii]) \log\mu([\iii]).
\end{equation*}
In addition, if $\mathsf{A} = (A_1,\ldots,A_N) \in GL_d(\R)^N$, then we define the $i$th \emph{Lyapunov exponent} of $\mu$ by setting
\begin{equation*}
  \lambda_{\mathsf{A}}(\alpha_i,\mu) = \lim_{n \to \infty} \tfrac{1}{n} \sum_{\iii \in \Sigma_n} \mu([\iii]) \log\alpha_i(A_\iii)
\end{equation*}
for all $i \in \{ 1,\ldots,d \}$. Furthermore, if $k \in \{ 0,\ldots,d-1 \}$ and $k<s<k+1$, then we define
\begin{equation*}
\begin{split}
  \lambda_{\mathsf{A}}(\fii^s,\mu) &= \lim_{n \to \infty} \tfrac{1}{n} \sum_{\iii \in \Sigma_n} \mu([\iii]) \log\fii^s(A_\iii) = \inf_{n \in \N} \tfrac{1}{n} \sum_{\iii \in \Sigma_n} \mu([\iii]) \log\fii^s(A_\iii) \\
  &= \lambda_{\mathsf{A}}(\alpha_1,\mu) + \cdots + \lambda_{\mathsf{A}}(\alpha_k,\mu) + (s-k)\lambda_{\mathsf{A}}(\alpha_{k+1},\mu).
\end{split}
\end{equation*}
Recalling \eqref{eq:cylinder2} and the fact that $\mu$ is invariant, the limits above exist and equal the infimums of the corresponding sequences by the standard theory of subadditive sequences.

An application of Jensen's inequality yields $P_{\mathsf{A}}(\fii^s) \ge h(\mu) + \lambda_{\mathsf{A}}(\fii^s,\mu)$ for all $\mu \in \MM_\sigma(\Sigma)$ and $s \ge 0$. Given ergodic $\mu\in \MM_\sigma(\Sigma)$ the number $s$ for which $h(\mu) + \lambda_{\mathsf{A}}(\fii^s,\mu) = 0$ is called the \emph{Lyapunov dimension} of $\mu$. A measure $\mu \in \MM_\sigma(\Sigma)$ is called an \emph{$\varphi^s$-equilibrium state} of $\mathsf{A}$ if it satisfies the following variational principle:
\begin{equation*}
  P_{\mathsf{A}}(\fii^s) = h(\mu) + \lambda_{\mathsf{A}}(\fii^s,\mu).
\end{equation*}
K\"{a}enm\"{a}ki \cite[Theorems 2.6 and 4.1]{Kaenmaki2004} proved that for each $\mathsf{A} \in GL_d(\R)^N$ and $s \ge 0$ there exists an ergodic $\fii^s$-equilibrium state of $\mathsf{A}$; see also \cite[Theorem 3.3]{KaenmakiVilppolainen2010}. The example of K\"aenm\"aki and Vilppolainen \cite[Example 6.2]{KaenmakiVilppolainen2010} shows that such an equilibrium state is not necessarily unique.

As a first observation on the structure of the set of all equilibrium states, we recall the following result; see Feng and K\"aenm\"aki \cite[Proposition 1.2]{FengKaenmaki2011}, Feng \cite[Theorem 5.5]{Feng2011}, and K\"aenm\"aki and Reeve \cite[Theorem A]{KaenmakiReeve2014}.

\begin{theorem} \label{thm:unique-eq}
  If $0 \le s \le d$ and $\mathsf{A} = (A_1,\ldots,A_N) \in GL_d(\R)^N$ is $\fii^s$-quasimultiplicative, then there exists a unique $\fii^s$-equilibrium state of $\mathsf{A}$ and it is fully supported.
\end{theorem}

By Lemma \ref{thm:irreducibility}, we thus have introduced a condition on matrices to guarantee the uniqueness of the equilibrium state.

Similarly to \eqref{eq:topo_def}, given $\mathsf{A} = (A_1,\ldots,A_N) \in GL_d(\R)^N$ and $s \ge 0$, we define
\begin{equation*}
  P_{\mathsf{A}}(\|\cdot\|^s) = \lim_{n \to \infty} \tfrac{1}{n} \log\sum_{\iii \in \Sigma_n} \| A_\iii \|^s
\end{equation*}
and call it the \emph{norm pressure} of $\mathsf{A}$. Note that $P_{\mathsf{A}}(\|\cdot\|^s) = P_{\mathsf{A}}(\fii^s)$ for all $0 \le s \le 1$. If $\mu \in \MM_\sigma(\Sigma)$, then we also set
\begin{equation*}
  \lambda_{\mathsf{A}}(\|\cdot\|^s,\mu) = s\lambda_{\mathsf{A}}(\alpha_1,\mu) = \lim_{n \to \infty} \tfrac{1}{n} \sum_{\iii \in \Sigma_n} \mu([\iii]) \log\|A_\iii\|^s.
\end{equation*}
It follows that $P_{\mathsf{A}}(\|\cdot\|^s) \ge h(\mu) + \lambda_{\mathsf{A}}(\|\cdot\|^s,\mu)$ for all $\mu \in \MM_\sigma(\Sigma)$ and $s \ge 0$. A measure $\mu \in \MM_\sigma(\Sigma)$ is called a \emph{$\|\cdot\|^s$-equilibrium state} of $\mathsf{A}$ if
\begin{equation*}
  P_{\mathsf{A}}(\|\cdot\|^s) = h(\mu) + \lambda_{\mathsf{A}}(\|\cdot\|^s,\mu).
\end{equation*}
The following theorem is proved by Feng and K\"aenm\"aki \cite[Theorem 1.7]{FengKaenmaki2011}.

\begin{theorem} \label{thm:FeKa}
  If $s\geq 0$ and $\mathsf{A} \in GL_d(\R)^N$, then there exist at most $d$ distinct ergodic $\|\cdot\|^s$-equilibrium states of $\mathsf{A}$ and they are all fully supported. Furthermore, if $\mathsf{A}$ is irreducible, then the equilibrium state is unique.
\end{theorem}

As remarked in \cite[\S 3]{FengKaenmaki2011}, it has the following corollary which further gives information about the structure of the set of all equilibrium states.

\begin{theorem} \label{thm:FeKa2}
  If $s \in (0,1) \cup (d-1,d) \cup \{ 0,\ldots,d \}$ and $\mathsf{A} \in GL_d(\R)^N$, then there exist at most $\binom{d}{s}$, if $s$ is an integer, and $d$, if otherwise, distinct ergodic $\fii^s$-equilibrium states of $\mathsf{A}$, and they are all fully supported.
\end{theorem}

Observe that together with the non-uniqueness observation \cite[Example 6.2]{KaenmakiVilppolainen2010} this immediately results in a two-dimensional version of Theorem \ref{thm:3d}.


\section{A geometric argument} \label{sec:alg}

In this section, we prove Theorem \ref{thm:intro-alg}. Recalling Lemma \ref{thm:irreducibility} and Theorem \ref{thm:unique-eq}, its proof follows after we have shown the following proposition.

\begin{proposition}\label{th:alg}
  Let $k \in \{ 0,\ldots,d-1 \}$ and $\emptyset \ne \mathcal{S} \subset GL_d(\R)$ be a $k$-irreducible and $(k+1)$-irreducible semigroup. If there exist nonzero $v_k, w_k \in \wedge^k \R^d$ and $v_{k+1}, w_{k+1} \in \wedge^{k+1} \R^d$ such that
  \begin{equation*}
    \la v_k, A^{\wedge k}w_k \ra_k \, \la v_{k+1}, A^{\wedge(k+1)}w_{k+1} \ra_{k+1} = 0,
  \end{equation*}
  for all $A \in \mathcal{S}$, then $\mathcal{S}$ is neither strongly $k$-irreducible nor strongly $(k+1)$-irreducible.
\end{proposition}

Before going into the proof, we remark that, instead of just being fully supported, the unique $\fii^s$-equilibrium state of $\mathsf{A}$ found in Theorem \ref{thm:intro-alg} satisfies a certain Gibbs property. The original formulation of Theorem \ref{thm:unique-eq} in \cite{KaenmakiReeve2014} implies this immediately. Since this fact is not needed in our considerations, we only state it in the following remark for possible future reference.

\begin{remark} \label{rem:gibbs}
  Let $k \in \{ 0,\ldots,d-1 \}$, $k<s<k+1$, and $\mathsf{A}\in GL_d(\mathbb{R})^N$. If $\mathsf{A}^{\wedge k}$ and $\mathsf{A}^{\wedge(k+1)}$ are both irreducible, and one of them is strongly irreducible, then there exists a unique $\varphi^s$-equilibrium state $\mu$ of $\mathsf{A}$ and it satisfies the following property: there exists $C \ge 1$ depending only on $\mathsf{A}$ and $s$ such that
  \begin{equation*}
    C^{-1}e^{-nP_{\mathsf{A}}(\fii^s)} \fii^s(A_\iii) \le \mu([\iii]) \le Ce^{-nP_{\mathsf{A}}(\fii^s)} \fii^s(A_\iii)
  \end{equation*}
  for all $\iii \in \Sigma_*$.
\end{remark}

We recall some elementary facts of algebraic geometry. Let us say that a function $p \colon GL_d(\mathbb{R}) \to \mathbb{R}$ is a \emph{polynomial} if it maps each matrix $A=[a_{ij}]_{i,j=1}^d$ to the same polynomial function of the variables\footnote{The conventional inclusion of the variable $(\det A)^{-1}$ is motivated by the fact that each entry of the matrix $A^{-1}$ is then a polynomial function of the matrix $A$. However, our interest in polynomials is essentially restricted to their use in defining the Zariski topology on $GL_d(\mathbb{R})$. In particular since $\det A$ is itself a polynomial function of the variables $a_{11},\ldots,a_{dd}$, the class of Zariski-closed sets which we consider is unaffected if the variable $(\det A)^{-1}$ is omitted.} $a_{11},\ldots,a_{dd}$ and $(\det A)^{-1}$. The \emph{Zariski topology} on $GL_d(\mathbb{R})$ is then defined to be the smallest topology in which every set of the form $\{A \in GL_d(\mathbb{R}) : p(A)=0\}$ is closed. The Zariski topology has the following important property, called the \emph{descending chain condition}: if $(Z_n)_{n=1}^\infty$ is a sequence of Zariski-closed sets such that $Z_{n+1} \subset Z_n$ for every $n \in \N$, then $(Z_n)_{n=1}^\infty$ is eventually constant. This property implies that a set is Zariski closed if and only if it is the intersection of the zero loci of a finite collection of polynomials. 

The following result is well-known, but we include a proof for the convenience of the reader who may be unfamiliar with algebraic geometry.

\begin{lemma}\label{le:lie}
Let $\mathcal{S}\subset GL_d(\mathbb{R})$ be a semigroup. Then the Zariski closure of $\mathcal{S}$ is a Lie group and has finitely many connected components.
\end{lemma}

\begin{proof}
In this proof all closures are taken with respect to the Zariski topology. To avoid triviality we assume $\mathcal{S}$ to be nonempty. We observe that if $A \in GL_d(\mathbb{R})$ then the preimage under left-multiplication by $A$ of a Zariski-closed set $Z=\bigcap_{p \in P} \{B \in GL_d(\mathbb{R})\colon p(B)=0\}$ is the set $A^{-1}Z=\bigcap_{p \in P} \{B \in GL_d(\mathbb{R})\colon p(AB)=0\}$, which is also Zariski-closed since $B \mapsto p(AB)$ is a polynomial. It follows that left multiplication by $A$ defines a Zariski continuous map from $GL_d(\mathbb{R})$ to itself. Since left-multiplication by $A^{-1}$ is also Zariski continuous it follows that left-multiplicaiton by any $A \in GL_d(\mathbb{R})$ induces a Zariski homeomorphism of $GL_d(\mathbb{R})$. Similar remarks apply to right multiplication, and we deduce that in particular $\overline{AX}=A\overline{X}$ and $\overline{XA}=\overline{X}A$ for every set $X \subset GL_d(\mathbb{R})$.

To begin the proof of the lemma let us show that the Zariski closure $\overline{\mathcal{S}}$ is a semigroup, that is, that $AB \in \overline{\mathcal{S}}$ for all $A,B \in \overline{\mathcal{S}}$. Obviously $\mathcal{S}A\subset \mathcal{S}$ for all $A \in \mathcal{S}$ and therefore $\overline{\mathcal{S}}A=\overline{\mathcal{S}A}\subset \overline{\mathcal{S}}$ for every $A \in \mathcal{S}$. Thus $\overline{\mathcal{S}}\mathcal{S}\subset \overline{\mathcal{S}}$. If $A \in \overline{S}$ we thus have in particular $A\mathcal{S}\subset \overline{\mathcal{S}}$ and therefore $A\overline{\mathcal{S}}=\overline{A\mathcal{S}}\subset \overline{\mathcal{S}}$. We conclude that for every $A,B \in \overline{\mathcal{S}}$ we have $AB \in \overline{\mathcal{S}}$ as claimed.

Let us now show that $\overline{\mathcal{S}}$ is in fact a group, for which it suffices to show that $A^{-1}\overline{\mathcal{S}}\subset \overline{\mathcal{S}}$ for every $A \in \overline{\mathcal{S}}$. Let $A \in \overline{\mathcal{S}}$. The sequence of sets $(A^n\overline{\mathcal{S}})_{n=1}^\infty$ is a descending sequence of Zariski-closed subsets of $GL_d(\mathbb{R})$ and by the descending chain condition it is eventually constant. Thus $A^{n}\overline{\mathcal{S}}= A^{n+1}\overline{\mathcal{S}}$ for some integer $n$, and by left multiplication by $A^{-n-1}$ we have $A^{-1}\overline{\mathcal{S}}=\overline{\mathcal{S}}$. It follows that $\overline{\mathcal{S}}$ is a group as claimed.

Since $\overline{\mathcal{S}}$ is Zariski closed it is the intersection of the zero loci of some finite collection of real polynomials. Such a set is well-known to have only finitely many connected components with respect to the standard topology (see e.g.\ \cite{BenedettiRisler1990}). It therefore remains only to show that $\overline{S}$ is a Lie group: but since $\overline{S}$ is closed in the Zariski topology it is closed in the standard topology, and by a celebrated theorem of von Neumann (\cite[Theorem 20.10]{Lee2013}) every subgroup of $GL_d(\mathbb{R})$ which is closed in the standard topology is a Lie group.
\end{proof}

We recall that every real Lie group admits a natural real-analytic structure and that every Lie group homomorphism is analytic. If $G$ is a Lie group then we let $G^\circ$ denote the identity component of $G$, that is, the unique connected component of $G$ which contains the identity element. Recall that, if $V$ is a real vector space, then $\Aut(V)$ is the group of all automorphisms of $V$, i.e.\ the set of all bijective linear transformations $V \to V$, together with functional composition as group operation. The set $\End(V)$ is the collection of all endomorphisms of $V$, i.e.\ the collection of all linear transformations $V \to V$. Recall that a representation of a Lie group $G$ on a finite-dimensional vector space $V$ over $\mathbb{R}$ is a group homomorphism $\rho \colon G \to \Aut(V)$. We say that a Lie group representation $\rho \colon G \to \Aut(V)$ is \emph{irreducible} if $\rho(G)$ is irreducible, that is, if there is no proper nonzero subspace of $V$ which is preserved by every element of $\rho(G)$.

\begin{lemma}\label{le:reps}
Let $G$ be a real Lie group and for each $i\in\{1,2\}$ let $V_i$ be a finite-dimensional real vector space with inner product $\langle \cdot,\cdot\rangle_{V_i}$, $\rho_i\colon G \to \Aut(V_i)$ an irreducible Lie group representation, and $u_i,v_i \in V_i$ nonzero vectors. Suppose that
\[
  \langle u_1,\rho_1(g)v_1\rangle_{V_1}\,\langle u_2,\rho_2(g)v_2\rangle_{V_2}=0
\]
for all $g \in G$. Then for each $i\in\{1,2\}$ there exists a nonzero vector $\hat{v}_i \in V_i$ such that $\langle u_i,\rho_i(g)\hat{v}_i\rangle_{V_i}=0$ for all $g \in G^\circ$.
\end{lemma}

\begin{proof}
Let us define $X_i=\{g \in G : \langle u_i,\rho_i(g)v_i\rangle_{V_i}=0\}$ for $i\in\{1,2\}$. Obviously $G=X_1\cup X_2$ and each $X_i$ is closed. We claim that neither $X_i$ is equal to $G$. Indeed, if $X_i=G$ then defining $U_i=\linspan\{\rho_i(g)v_i \colon g\in G\}\subset V_i$ we find that $\rho_i(g)U_i=U_i$ for every $g \in G$. Since $U_i$ is contained in the orthogonal complement of $u_i\neq 0$ it is a proper subspace of $V_i$, and since $U_i$ contains $v_i \neq 0$ it is not the zero subspace. It follows that the representation $\rho_i \colon G \to \Aut(V_i)$ is reducible, contradicting the hypotheses of the lemma. This completes the proof of the claim.

Since $G=X_1 \cup X_2$ and $X_1$ is a closed proper subset of $G$, the set $X_2$ contains the nonempty open set $G \setminus X_1$. Similarly $X_1$ contains the nonempty open set $G \setminus X_2$. For each $i\in\{1,2\}$ let $\ell_i \colon \mathrm{End}(V_i) \to \mathbb{R}$ be the linear mapping given by $\ell_i(A)=\langle u_i,Av_i\rangle_{V_i}$ for every $A \in \mathrm{End}(V_i)$. The composition $\ell_i \circ \rho_i \colon G \to \mathbb{R}$ is consequently real analytic and is zero on $X_i$, which contains a nonempty open set. By analyticity it therefore follows that $\ell_i \circ \rho_i$ is zero on a connected component $H_i$ of $G$. 

For each $i\in\{1,2\}$ choose any $g_i \in H_i$. Since right multiplication by $g_i$ induces a homeomorphism of $G$ the set $G^\circ g_i$ is a closed and open subset of $G$, and since $G^\circ$ contains the identity, $G^\circ g_i$ contains $g_i$. It follows that $G^\circ g_i=H_i$ for all $i\in\{1,2\}$. In particular, we have
\[
  \langle u_i,\rho_i(g)\rho_i(g_i)v_i\rangle_{V_i}=\langle u_i,\rho_i(gg_i)v_i\rangle_{V_i}=0
\]
for every $g \in G^\circ$. Taking $\hat{v}_i=\rho_i(g_i)v_i \in V_i$ completes the proof of the lemma.
\end{proof}

\begin{proof}[Proof of Proposition \ref{th:alg}]
Let $G$ be the Zariski closure of $\mathcal{S}$, which by Lemma \ref{le:lie} is a Lie subgroup of $GL_d(\mathbb{R})$ with finitely many connected components. The set of all $A \in GL_d(\mathbb{R})$ such that
\begin{equation}\label{eq:prodzero}\la v_k, A^{\wedge k}w_k \ra_k \, \la v_{k+1}, A^{\wedge(k+1)}w_{k+1} \ra_{k+1} = 0\end{equation}
is the zero set of a polynomial function $GL_d(\mathbb{R}) \to \mathbb{R}$ and hence is Zariski closed. Since this set contains $\mathcal{S}$, it contains the Zariski closure of $\mathcal{S}$ and therefore every $A$ in $G$ satisfies \eqref{eq:prodzero}.

Define Lie group representations $\rho_1 \colon G \to \Aut(\wedge^k\mathbb{R}^d)$ and $\rho_2 \colon G \to \Aut(\wedge^{k+1}\mathbb{R}^d)$ by $A \mapsto A^{\wedge k}$ and $A\mapsto A^{\wedge(k+1)}$, respectively. If $\rho_1$ were reducible then $G$, and hence $\mathcal{S}$, would preserve a proper nonzero linear subspace of $\wedge^k\mathbb{R}^d$, contradicting the hypothesis that $\mathcal{S}$ is $k$-irreducible. It follows that $\rho_1$ is an irreducible representation, and similarly $\rho_2$ is irreducible since $\mathcal{S}$ is $(k+1)$-irreducible.  Lemma \ref{le:reps} thus implies that there exist nonzero $\hat{w}_k \in \wedge^k\mathbb{R}^d$ and $\hat{w}_{k+1}\in \wedge^{k+1}\mathbb{R}^d$ such that $\la v_k, A^{\wedge k}\hat{w}_k \ra_k=0$ and $\la v_{k+1}, A^{\wedge(k+1)}\hat{w}_{k+1} \ra_{k+1} = 0$ for all $A \in G^\circ$.

Let us define linear subspaces $U_1$ of $ \wedge^k\mathbb{R}^d$ and $U_2$ of  $\wedge^{k+1}\mathbb{R}^d$ by
\[
  U_1=\linspan\{A^{\wedge k}\hat{w}_k : A \in G^\circ\} \quad \text{and} \quad U_2=\linspan\{A^{\wedge(k+1)}\hat{w}_{k+1} : A \in G^\circ\}.
\]
Clearly $U_1$ is a proper subspace of $\wedge^k\mathbb{R}^d$ since it is contained in the orthogonal complement of $v_k$, and similarly $U_2$ is a proper subspace of $\wedge^{k+1}\mathbb{R}^d$. It is also clear that $A^{\wedge k}U_1\subset U_1$ and $A^{\wedge (k+1)}U_2\subset U_2$ for every $A \in G^\circ$, which by the invertibility of $A$ implies 
$A^{\wedge k}U_1=U_1$ and $A^{\wedge (k+1)}U_2=U_2$ for every $A \in G^\circ$. 

We claim that if $A,B \in G$ belong to the same connected component of $G$ then $A^{\wedge k}U_1=B^{\wedge k}U_1$. To see this fix $B \in G$ and note that $BG^\circ$ is a closed and open subset of $G$ which contains $B$; thus, $BG^\circ$ is the connected component of $B$. In particular if $A$ is in the same component as $B$ then $A \in BG^\circ$ and therefore $B^{-1}A \in G^\circ$. It follows that $(B^{-1}A)^{\wedge k}U_1=U_1$ and therefore $A^{\wedge k}U_1=B^{\wedge k}U_1$. In the same manner it follows that if $A,B \in G$ belong to the same connected component of $G$ then $A^{\wedge (k+1)}U_2=B^{\wedge (k+1)}U_2$.

Let us define
\[
  X_1=\bigcup_{A \in G} A^{\wedge k}U_1 \quad \text{and} \quad X_2=\bigcup_{A \in G} A^{\wedge (k+1)}U_2.
\]
Obviously $A^{\wedge k}X_1=X_1$ and $A^{\wedge(k+1)}X_2=X_2$ for every $A \in G$. By the preceding claim, the number of distinct subspaces $A^{\wedge k}U_1$ as $A$ ranges over $G$ is no greater than the number of connected components of $G$, which is finite. Thus $X_1$ is equal to the union of finitely many proper nonzero linear subspaces of $\wedge^k\mathbb{R}^d$. Similarly $X_2$ is a finite union of proper nonzero subspaces of $\wedge^{k+1}\mathbb{R}^d$. Since $X_1$ is preserved by $A^{\wedge k}$ for every $A \in \mathcal{S}\subset G$ it follows that $\mathcal{S}$ is not strongly $k$-irreducible, and by similar consideration of $X_2$, $\mathcal{S}$ is not strongly $(k+1)$-irreducible. The proof of the proposition is complete.
\end{proof}


\section{Permutation matrices} \label{sec:permutation}

In this section, we prove Theorem \ref{thm:intro-perm} by showing that the $\varphi^s$-equilibrium states of generalised permutation matrices can be understood via the $\|\cdot\|$-equilibrium states of certain auxiliary matrices. Recall that $P_d(\R) \subset GL_d(\mathbb{R})$ is the group of generalised permutation matrices and that $A \in P_d(\R)$ if every row and every column of $A$ has exactly one nonzero entry. Observe that $A$ is a generalised permutation matrix if and only if there exist a bijection $\pi_A \colon \{1,\ldots,d\} \to \{1,\ldots,d\}$ and nonzero real numbers $a_1,\ldots,a_d$ such that $Ae_i=a_ie_{\pi_A(i)}$ for all $i \in \{1,\ldots,d\}$. If $A, B \in P_d(\R)$, then clearly
\begin{equation} \label{eq:permutation-product}
  ABe_i=a_{\pi_B(i)}b_ie_{(\pi_A \circ \pi_B)(i)}
\end{equation}
for all $i\in\{1,\ldots,d\}$.

Fix $k \in \{ 0,\ldots,d-1 \}$,  $k < s < k+1$, and $d'=(d-k)\binom{d}{k}$. Let $S_{k,d}$ be the set of all $k$-combinations of $\{ 1,\ldots,d \}$. Denote the standard basis of $\R^d$ by $\{ e_1,\ldots,e_d \}$ and let the standard basis of $\R^{d'}$ be relabelled as
\begin{equation*}
  \{ e_{S,i} : S \in S_{k,d} \text{ and } i \in \{ 1,\ldots,d \} \setminus S \}.
\end{equation*}
If $A \in P_d(\R)$, then let $\pi_A \colon \{1,\ldots,d\} \to \{1,\ldots,d\}$ be a bijection and $a_1,\ldots,a_d$ nonzero real numbers such that $Ae_i = a_ie_{\pi_A(i)}$ for all $i \in \{ 1,\ldots,d \}$. Setting
\begin{equation*}
  \mathfrak{h}_s(A)e_{S,i} = \biggl(\prod_{j \in S}|a_j|\biggr)|a_i|^{s-k} e_{\pi_A(S),\pi_A(i)}
\end{equation*}
defines a mapping $\mathfrak{h}_s \colon P_d(\R) \to P_{d'}(\R)$.

\begin{lemma} \label{thm:ddim-permutation}
  The mapping $\mathfrak{h}_s \colon P_d(\R) \to P_{d'}(\R)$ is a homomorphism and $\varphi^s(A)=\|\mathfrak{h}_s(A)\|$ for all $A \in P_d(\R)$.
\end{lemma}

\begin{proof}
  The norm of a generalised permutation matrix $A \in P_d(\R)$ is simply the maximum of the absolute values of its entries, and the singular values are the absolute values of the nonzero entries listed in decreasing order. It follows that
  \[
    \varphi^s(A) = \max_{\substack{1 \leq i_1,\ldots,i_{k+1}\leq d\\ i_j \neq i_{\ell}}} |a_{i_1}\cdots a_{i_k}| |a_{i_{k+1}}|^{s-k}=\max_{\substack{S \in S_{k,d}\\ i \in \{ 1,\ldots,d \} \setminus S}} \biggl(\prod_{j\in S}|a_j|\biggr)|a_i|^{s-k}=\|\mathfrak{h}_s(A)\|.
  \]
  To check that $\mathfrak{h}_s$ is a homomorphism, let $A, B \in P_d(\R)$ respectively satisfy $Ae_i=a_ie_{\pi_A(i)}$ and $Be_i=b_ie_{\pi_B(i)}$ for all $i \in \{1,\ldots,d\}$. For every $S \subset \{1,\ldots,d\}$ and $i \in \{1,\ldots,d\}\setminus S$ we have
  \[
    \mathfrak{h}_s(B)e_{S,i} =\biggl(\prod_{j \in S}|b_j|\biggr)|b_i|^{s-k} e_{\pi_B(S),\pi_B(i)}
  \]
  and similarly for $A$. Therefore,
  \[
    \mathfrak{h}_s(A)\mathfrak{h}_s(B)e_{S,i}=\biggl(\prod_{j' \in \pi_B(S)} |a_{j'}|\biggr)\biggl(\prod_{j \in S} |b_j|\biggr)|a_{\pi_B(i)}|^{s-k}|b_i|^{s-k} e_{(\pi_A\circ \pi_B)(S),(\pi_A\circ \pi_B)(i)}.
  \]
  On the other hand, by \eqref{eq:permutation-product}, we have
  \[
    \mathfrak{h}_s(AB)e_{S,i}=\biggl(\prod_{j \in S} |a_{\pi_B(j)}b_j|\biggr)|a_{\pi_B(i)}b_i|^{s-k} e_{(\pi_A\circ \pi_B)(S),(\pi_A\circ \pi_B)(i)}
  \]
  and hence, $\mathfrak{h}_s(AB)=\mathfrak{h}_s(A)\mathfrak{h}_s(B)$ as required.
\end{proof}

With the auxiliary matrices given by $\mathfrak{h}_s$, we may now apply Theorem \ref{thm:FeKa}.

\begin{proposition}\label{co:permutation-equilibrium}
  Let $k \in \{ 0,\ldots,d-1 \}$, $k < s < k+1$, and $\mathsf{A}=(A_1,\ldots,A_N)\in P_d(\mathbb{R})^N$. Then $\mu$ is a $\varphi^s$-equilibrium state of $\mathsf{A}$ if and only if it is a $\|\cdot\|$-equilibrium state of $\mathfrak{h}_s(\mathsf{A}) = (\mathfrak{h}_s(A_1),\ldots,\mathfrak{h}_s(A_N))$. In particular, there are at most $(d-k)\binom{d}{k}$ distinct ergodic $\varphi^s$-equilibrium states of $\mathsf{A}$ and they are all fully supported.
\end{proposition}

\begin{proof}
  By Lemma \ref{thm:ddim-permutation}, we have 
  \[
    P_{\mathsf{A}}(\fii^s) = P_{\mathfrak{h}_s(\mathsf{A})}(\|\cdot\|) \quad \text{and} \quad \lambda_{\mathsf{A}}(\fii^s,\mu) = \lambda_{\mathfrak{h}_s(\mathsf{A})}(\|\cdot\|,\mu)
  \]
  for all $\mu \in \MM_\sigma(\Sigma)$. This implies that the two sets of equilibrium states are identical. The number of ergodic $\varphi^s$-equilibrium states of $\mathsf{A}$ is therefore equal to the number of ergodic $\|\cdot\|$-equilibrium states of $\mathfrak{h}_s(\mathsf{A})$. By Theorem \ref{thm:FeKa}, this number is bounded above by the dimension $(d-k)\binom{d}{k}$ of the matrices $\mathfrak{h}_s(A_i)$. Moreover, each of the ergodic equilibrium state is fully supported.
\end{proof}

To finish the proof of Theorem \ref{thm:intro-perm} it only requires to show that the upper bound $(d-k)\binom{d}{k}$ found in Proposition \ref{co:permutation-equilibrium} can be attained. This is done in the following proposition.

\begin{proposition}\label{pr:max-e-states}
  Let $k \in \{ 0,\ldots,d-1 \}$, $k<s<k+1$, and $\mathsf{A}=(A_1,\ldots,A_d)\in P_d(\mathbb{R})^d$ be such that $A_ie_i=2e_i$ and $A_ie_j=e_j$ for all $i,j \in \{ 1,\ldots,d \}$ with $i \ne j$. Then the number of distinct ergodic $\varphi^s$-equilibrium states of $\mathsf{A}$ is precisely $(d-k)\binom{d}{k}$.
\end{proposition}

\begin{proof}
  Let $d'=(d-k)\binom{d}{k}$. By Proposition \ref{co:permutation-equilibrium}, it is sufficient to show that the $d$-tuple of $d'$-dimensional matrices  $\mathfrak{h}_s(\mathsf{A}) = (\mathfrak{h}_s(A_1),\ldots,\mathfrak{h}_s(A_d))$ has precisely $d'$ distinct ergodic $\|\cdot\|$-equilibrium states. Observe that the matrix $\mathfrak{h}_s(A_i)$ satisfies
  \begin{equation*}
    \mathfrak{h}_s(A_i)e_{S,j} =
    \begin{cases}
      e_{S,j}, &\text{if } i \notin S \cup \{j\}, \\
      2e_{S,j}, &\text{if } i \in S, \\
      2^{s-k}e_{S,j}, &\text{if } i=j,
    \end{cases}
  \end{equation*}
  for all $i,j \in \{ 1,\ldots,d \}$. In particular, each $\mathfrak{h}_s(A_i)$ is diagonal. Therefore, by \cite[Theorem 1.7]{FengKaenmaki2011}, an ergodic measure $\mu$ is a $\|\cdot\|$-equilibrium state of $\mathfrak{h}_s(\mathsf{A})$ if and only if there exists a basis element $e_{S,j}$ such that $\mu$ is a $\|\cdot\|$-equilibrium state of the $d'$-tuple of $1\times 1$ matrices $\mathfrak{h}_s(\mathsf{A})e_{S,j} = (|\mathfrak{h}_s(A_1)e_{S,j}|,\ldots,|\mathfrak{h}_s(A_d)e_{S,j}|)$ and such that the norm pressure of $\mathfrak{h}_s(\mathsf{A})e_{S,j}$ is maximal with respect to the choice of $(S,j)$.

  To prove the proposition it therefore suffices to show that if $(S_1,i_1) \neq (S_2,i_2)$ then the norm pressures of $\mathfrak{h}_s(\mathsf{A})e_{S_1,i_1}$ and $\mathfrak{h}_s(\mathsf{A})e_{S_2,i_2}$ are the same but their $\|\cdot\|$-equilibrium states are different. Indeed, for $j \in \{1,2\}$ the norm pressures are simply given by
  \begin{align*}
    P_{\mathfrak{h}_s(\mathsf{A})e_{S_j,i_j}}(\|\cdot\|) &= \log \biggl(\sum_{i=1}^d |\mathfrak{h}_s(A_i)e_{S_j,i_j}|\biggr) \\
    &=\log(2k+2^{s-k}+d-k-1)=\log(2^{s-k}+d+k-1)
  \end{align*}
  since exactly $k$ summands equal $2$, exactly one summand equals $2^{s-k}$, and the remaining $d-k-1$ summands equal $1$.  The two norm pressures are therefore equal as desired.
  Now let $\mu_1$ and $\mu_2$ denote the respective $\|\cdot\|$-equilibrium states corresponding to the distinct pairs $(S_1,i_1)$ and $(S_2,i_2)$. For $j\in \{1,2\}$ the measure $\mu_j$ is the unique Bernoulli measure on $\Sigma$ such that 
  \[
    \mu_j([i])=\frac{|\mathfrak{h}_s(A_i)e_{S_j,i_j}|}{2^{s-k}+d+k-1}
  \]
  for every $i\in\{1,\ldots,d\}$. If $i_1 \neq i_2$, then $|\mathfrak{h}_s(A_{i_1})e_{S_1,i_1}|=2^{s-k}$ but $|\mathfrak{h}_s(A_{i_1})e_{S_2,i_2}|$ is either $1$ or $2$ depending on whether $i_1 \in S_2$, so $\mu_1([i_1])\neq \mu_2([i_1])$. On the other hand, if $S_1 \neq S_2$ then for $i \in S_1 \triangle S_2$ one of the values $|\mathfrak{h}_s(A_{i})e_{S_1,i_1}|$ and  $|\mathfrak{h}_s(A_{i})e_{S_2,i_2}|$ equals $2$ and the other equals either $1$ or $2^{s-k}$, and therefore $\mu_1([i])\neq \mu_2([i])$. We conclude that the number of distinct ergodic $\|\cdot\|$-equilibrium states of $\mathfrak{h}_s(\mathsf{A})$ equals the number of distinct basis elements $(S,j)$ which is of course precisely $d'$.
\end{proof}

We note that the $\varphi^s$-equilibrium states of $\mathsf{A} \in GL_d(\mathbb{R})^N$ cannot in general be represented as $\|\cdot\|^t$-equilibrium states of  a collection of auxiliary matrices of dimension strictly less than $(d-k)\binom{d}{k}$. If this were the case then the maximum number of $\varphi^s$-equilibrium states of $\mathsf{A}$ would have to be strictly less than $(d-k)\binom{d}{k}$, contradicting Proposition \ref{pr:max-e-states}.


\section{Block upper-triangular matrices} \label{sec:block-upper}

In this section, we prove Theorem \ref{thm:intro-block}. To that end, we will first state and prove a technical lemma which allows us to estimate the singular value function of a block upper triangular matrix by the singular value function of the corresponding block diagonal matrix.

\begin{lemma} \label{thm:triangular-diagonal}
  If $0<s<d$ and $A, A' \in GL_d(\mathbb{R})$ are such that
  \begin{equation*}
    A =
    \begin{pmatrix}
      B & C \\
      0 & D
    \end{pmatrix}
    \quad \text{and} \quad
    A' =
    \begin{pmatrix}
      B & 0 \\
      0 & D
    \end{pmatrix}
  \end{equation*}
  for some square matrices $B$ and $D$, then $\fii^s(A) \ge \fii^s(A')$.
\end{lemma}

\begin{proof}
  By the singular value decomposition there exist isometries $U_1,V_1,U_2,V_2$ and diagonal matrices $G_1,G_2$ such that
  \begin{equation*}
    U_1BV_1 = G_1 \quad \text{and} \quad U_2DV_2 = G_2.
  \end{equation*}
  Since $\varphi^s(UAV)=\varphi^s(A)$ whenever $U$ and $V$ are isometries it follows that
  \[
    \fii^s(A) = \fii^s\biggl(
    \begin{pmatrix}
      U_1 & 0 \\
      0 & U_2
    \end{pmatrix}
    \begin{pmatrix}
      B & C \\
      0 & D
    \end{pmatrix}
    \begin{pmatrix}
      V_1 & 0 \\
      0 & V_2
    \end{pmatrix}\biggr)
    = \fii^s
    \begin{pmatrix}
      G_1 & U_1CV_2 \\
      0 & G_2
    \end{pmatrix},
  \]
  where we note that the final matrix is upper triangular, and
  \[
    \fii^s(A') = \fii^s\biggl(
    \begin{pmatrix}
      U_1 & 0 \\
      0 & U_2
    \end{pmatrix}
    \begin{pmatrix}
      B & 0 \\
      0 & D
    \end{pmatrix}
    \begin{pmatrix}
      V_1 & 0 \\
      0 & V_2
    \end{pmatrix}\biggr)
    = \fii^s
    \begin{pmatrix}
      G_1 & 0 \\
      0 & G_2
    \end{pmatrix},
  \]
  where we observe the final matrix to be diagonal. We therefore lose no generality by assuming $A$ to be upper triangular and $A'$ diagonal. If $k \in \{ 0,\ldots,d-1 \}$ is such that $k<s<k+1$, then, by \eqref{eq:svd-formula}, we have
  \begin{equation*}
    \varphi^s(B)=\max\biggl\{\biggl(\prod_{i=1}^{k} \|Bu_i\|\biggr)\|Bu_{k+1}\|^{s-k} : u_1,\ldots,u_{k+1} \in S^{d-1} \text{ are pairwise orthogonal} \biggr\}
  \end{equation*}
  for all $B \in GL_d(\R)$. For $A'$ it is clear that this maximum is attained by taking $u_1,\ldots,u_{k+1}$ to be an appropriate subset of the standard basis. In this case, we clearly have $\|Au_i\|\geq \|A' u_i\|$ for every $i \in \{ 1,\ldots,k+1 \}$, and therefore $\fii^s(A) \ge \fii^s(A')$ as claimed. 
\end{proof}

Theorem \ref{thm:intro-block} may be obtained by a repeated application of the following proposition. Its proof is based on the continuity of the singular value pressure and the previous lemma.

\begin{proposition}\label{thm:2-block}
  Let $l \in \{ 1,\ldots,d-1 \}$ and $\mathsf{A}=(A_1,\ldots,A_N) \in GL_d(\mathbb{R})^N$ be such that
  \[
    A_i =
    \begin{pmatrix}
      B_i & C_i \\
      0   & D_i
    \end{pmatrix}
  \]
  for every $i \in \{ 1,\ldots,N \}$, where $B_i \in GL_l(\mathbb{R})$, $D_i \in GL_{d-l}(\mathbb{R})$, and the matrices $C_i$ have dimension $l \times (d-l)$. If $\mathsf{A}' = (A_1',\ldots,A_N') \in GL_d(\mathbb{R})$ is such that
  \[
    A_i' =
    \begin{pmatrix}
      B_i & 0 \\
      0   & D_i
    \end{pmatrix}
  \]
  for all $i \in \{1,\ldots,N\}$, then the set of all $\varphi^s$-equilibrium states of $\mathsf{A}$ is precisely the set of all $\varphi^s$-equilibrium states of $\mathsf{A}'$ for all $0<s<d$.
\end{proposition}

\begin{proof}
  Let $\mu$ be a $\fii^s$-equilibrium state of $\mathsf{A}$. For each $\varepsilon>0$ let us define $\mathsf{A}^\varepsilon=(A_1^\varepsilon,\ldots,A_N^\varepsilon)$ by setting
  \[
    A_i^\varepsilon =
    \begin{pmatrix}
      B_i & \varepsilon C_i \\
      0   & D_i
    \end{pmatrix}
    =
    \begin{pmatrix}
      \varepsilon^{1/2} I & 0 \\
      0 & \varepsilon^{-1/2} I
    \end{pmatrix}
    \begin{pmatrix}
      B_i & C_i \\
      0   & D_i
    \end{pmatrix}
    \begin{pmatrix}
      \varepsilon^{-1/2} I & 0 \\
      0 & \varepsilon^{1/2} I
    \end{pmatrix}
  \]
  for all $i \in \{ 1,\ldots,N \}$, where $I$ denotes the $l \times l $ or $(d-l) \times (d-l)$ identity matrix as appropriate. Since $\mathsf{A}^\varepsilon$ is conjugate to $\mathsf{A}$ it has the same singular value pressure and the same $\varphi^s$-equilibrium states. Thus $\mu$ is a $\varphi^s$-equilibrium state of $\mathsf{A}^\varepsilon$ for all $\eps > 0$. Observe that the function
  \[
    \mathsf{B} \mapsto \lambda_{\mathsf{B}}(\fii^s,\mu) =  \inf_{n \in \N} \tfrac{1}{n} \sum_{\iii \in \Sigma_n} \mu([\iii])\log\fii^s(B_\iii)
  \]
  defined on $GL_d(\R)^N$ is an infimum of continuous functions and hence upper semi-continuous. It follows that
  \begin{equation*}
    \lambda_{\mathsf{A}'}(\fii^s,\mu) \ge \limsup_{\eps \downarrow 0} \lambda_{\mathsf{A}^\eps}(\fii^s,\mu)
  \end{equation*}
  and hence
  \begin{equation*}
    h(\mu) + \lambda_{\mathsf{A}'}(\fii^s,\mu) \ge \limsup_{\eps \downarrow 0} P_{\mathsf{A}^\eps}(\fii^s) = P_{\mathsf{A}'}(\fii^s)
  \end{equation*}
  by Theorem \ref{le:cty}. Therefore $\mu$ is a $\fii^s$-equilibrium state of $\mathsf{A}'$.

  To show the other direction, let $\mu$ be a $\fii^s$-equilibrium state of $\mathsf{A}'$. Recall that $P_{\mathsf{A}}(\fii^s) = P_{\mathsf{A}^\eps}(\fii^s)$ by conjugacy and $\lim_{\eps \downarrow 0} P_{\mathsf{A}^\eps}(\fii^s) = P_{\mathsf{A}'}(\fii^s)$ by Theorem \ref{le:cty}. Therefore, by Lemma \ref{thm:triangular-diagonal}, we have
  \begin{equation*}
    P_{\mathsf{A}}(\fii^s) = P_{\mathsf{A}'}(\fii^s) = h(\mu) + \lambda_{\mathsf{A}'}(\fii^s,\mu) \le h(\mu) + \lambda_{\mathsf{A}}(\fii^s,\mu)\leq P_{\mathsf{A}}(\fii^s)
  \end{equation*}
  and $\mu$ is a $\fii^s$-equilibrium state of $\mathsf{A}$.
\end{proof}


\section{The three-dimensional case} \label{sec:3d}

In this section, we give the proof of Theorem \ref{thm:3d}, which proceeds through a series of cases. If $0 < s \le 1$ or $2 \le s < 3$, then the claim follows immediately from Theorem \ref{thm:FeKa2} and Proposition \ref{pr:max-e-states}, so we assume $1<s<2$ throughout the section. Furthermore, if $\mathsf{A} = (A_1,\ldots,A_N) \in GL_d(\R)^N$ is $\varphi^s$-quasimultiplicative, then the result follows from Theorem \ref{thm:unique-eq}, so we assume throughout the section that this is not the case. Observe that if $\mathsf{A}$ is strongly irreducible then $\mathsf{A}^{\wedge 2}$ must be irreducible by Lemma \ref{thm:irr_equiv2} and therefore Theorem \ref{th:alg} and Lemma \ref{thm:irreducibility} show that $\mathsf{A}$ must be $\varphi^s$-quasimultiplicative. Thus, our standing assumptions in this section imply that $\mathsf{A}$ cannot be strongly irreducible. In this section, we let $\rho(A)$ denote the spectral radius of the matrix $A$.

\subsection{The irreducible case} \label{sec:irred-case}
We first consider the case in which $\mathsf{A}$ is irreducible but not $\varphi^s$-quasimul\-tiplicative. We begin our analysis with a pair of lemmas.

\begin{lemma}\label{thm:boundedness}
  If $\mathsf{A} = (A_1,\ldots,A_N) \in GL_3(\R)^N$ is irreducible and $\rho(A_\iii)=1$ for every $\iii \in \Sigma_*$, then there exists $M>1$ such that  $\|A_\iii\|\leq M$ for every $\iii \in \Sigma_*$.
\end{lemma}
\begin{proof}
  By a well-known theorem of Berger and Wang (see e.g.\ \cite{BeWa92, El95, Bo03}), if $\rho(A_\iii)=1$ for every $\iii \in \Sigma_*$ then it follows immediately that the joint spectral radius
  \[\lim_{n \to \infty} \max_{|\iii|=n} \left\|A_\iii\right\|^{1/n}\]
  is equal to $1$. This implies the boundedness of the set $\{A_\iii \colon \iii \in \Sigma_*\}$ by e.g.\ \cite[Theorem 2.1]{Ju09}.
\end{proof}

\begin{lemma} \label{thm:equal-eigenvalues}
  If $\mathsf{A} = (A_1,\ldots,A_N) \in GL_3(\R)^N$ is irreducible, and for every $\iii \in \Sigma_*$ the three eigenvalues of $A_\iii$ all have the same modulus, then $\mathsf{A}$ is $\fii^s$-quasimultiplicative for all $0 < s < 3$.
\end{lemma}

\begin{proof}
  Fix $0 < s < 3$. By replacing each $A_i$ with $|\det(A_i)|^{-1/d}A_i$ if necessary, we may assume without loss of generality that every  $A_\iii$ has determinant $\pm 1$ and therefore every $A_\iii$ has all eigenvalues of modulus 1. In particular $\rho(A_\iii)=\rho(A_\iii^{-1})=1$ for every $\iii \in \Sigma_*$. By the previous lemma it follows that $\{A_\iii \colon \iii \in \Sigma_*\}$ is bounded, and applying this reasoning to $(A_1^{-1},\ldots,A_N^{-1})$ it follows that $\{A_\iii^{-1} \colon \iii \in \Sigma_*\}$ is bounded also. Thus there exists a constant $M>1$ such that for every $\iii \in \Sigma_*$
  \[
    M^{-1}\leq \|A_\iii^{-1}\|^{-1} =\alpha_d(A_\iii) \leq \alpha_{d-1}(A_\iii) \leq \cdots \leq \alpha_1(A_\iii) = \|A_\iii\|\leq M
  \]
  and consequently, $M^{-s} \leq \fii^s(A_\iii)\leq M^s$ for all $\iii \in \Sigma_*$.  It follows that for every $\iii,\jjj \in \Sigma_*$
  \[
    \varphi^s(A_\iii)\varphi^s(A_\jjj) \leq M^{2s} \leq M^{3s}\varphi^s(A_{\iii\jjj})
  \]
  and therefore $\mathsf{A}$ is $\fii^s$-quasimultiplicative as claimed.
  \end{proof}

If $u \in \mathbb{R}^3$ is a nonzero vector, let us write $\overline{u}$ for the one-dimensional subspace generated by $u$. We may now demonstrate that in three dimensions the irreducible but not $\varphi^s$-quasimultiplicative case may be reduced to the case of generalised permutation matrices studied in \S \ref{sec:permutation}.

\begin{proposition} \label{thm:3dim-permutation}
  Let $1<s<2$ and $\mathsf{A}=(A_1,\ldots,A_N)\in GL_3(\mathbb{R})^N$ be irreducible such that $\mathsf{A}$ is not $\fii^s$-quasi\-multiplicative. Then there exists a basis of $\mathbb{R}^3$ with respect to which $\mathsf{A} \in P_3(\R)^N$.
\end{proposition}

\begin{proof}
  Let $G$ denote the group generated by $A_1,\ldots,A_N$. Clearly it suffices to find a basis in which every element of $G$ is a generalised permutation matrix. If for every $A \in G$ all of the eigenvalues of $A$ are equal in modulus then, by Lemma \ref{thm:equal-eigenvalues}, $\mathsf{A}$ is $\fii^s$-quasimultiplicative for all $0 < s < 3$. Since this contradicts the assumption we conclude that there exists $A \in G$ whose eigenvalues are not all equal in modulus.

  It is sufficient to prove that there exist linearly independent vectors $v_1,v_2,v_3 \in \mathbb{R}^3$ such that $\{\overline{v}_1,\overline{v}_2,\overline{v}_3\}$ is preserved by $G$. Taking these vectors to be a new basis yields the result. Since $\mathsf{A}$ is not strongly irreducible (as pointed out in the beginning of this section) there exists a proper nontrivial subspace $V$ of $\mathbb{R}^3$ such that the orbit of $V$ under the semigroup generated by $\mathsf{A}$, and therefore under $G$, is finite. We will prove the proposition in the case where $V$ is a one-dimensional space. If instead $V$ is two-dimensional then the one-dimensional space $V^\perp$ has a finite orbit under the irreducible group $G^T=\{B^T\colon B \in G\}$ and so the conclusion of the proposition applies to $G^T$. Obviously if $G^T$ is simultaneously similar to a group of generalised permutation matrices then so is $G$, and thus no generality is lost by the assumption $\dim V=1$.

  Either $A$ or $A^{-1}$ has the property that its leading eigenspace is one-dimensional, and without loss of generality we assume this to be $A$. Let $\overline{u}$ be the leading eigenspace of $A$ and $P$ the $A$-invariant plane generated by its other two (generalised) eigenvectors. Let $\{\overline{v}_1,\ldots,\overline{v}_n\}$ denote the orbit of $V$, which is a finite set of distinct one-dimensional subspaces of $\mathbb{R}^3$ which is preserved by $G$. We will show that necessarily $n=3$.

  Let us first show that $n \geq 3$. We also show that precisely one of the $\overline{v}_i$'s is transverse to $P$. Indeed, if $n \leq 2$ then the linear span of the subspaces $\overline{v}_i$ would be a nontrivial proper subspace of $\mathbb{R}^3$ which is invariant under $G$, contradicting irreducibility. Similarly, if every $\overline{v}_i$ belongs to $P$, then the linear span of the $\overline{v}_i$'s is a proper nontrivial $G$-invariant subspace of $\mathbb{R}^3$, so at least one $\overline{v}_i$ is transverse to $P$. Now if $\overline{v}_i$ is transverse to $P$ then $\lim_{n\to\infty} \overline{A^nv}_i=\overline{u}$ since $\overline{u}$ is the leading eigenspace of $A$. It follows that if $\overline{v}_i$ is transverse to $P$ then either $\overline{v}_i=\overline{u}$ or the orbit of $\overline{v}_i$ under $G$ is infinite, but the latter is a contradiction. We conclude that exactly one $\overline{v}_i$ is equal to $\overline{u}$ and the remainder are subspaces of $P$. Without loss of generality we take $\overline{v}_1=\overline{u}$.

  Let us then show that $n \leq 3$. By irreducibility, the span of the subspaces $\overline{v}_2,\ldots,\overline{v}_n$ is not invariant under $G$, so there exist $B \in G$ and $\overline{v}_\ell$ with $\ell>1$ such that $\overline{Bv}_\ell=\overline{v}_1$. In particular $\overline{v}_\ell$ is a subspace of $P$. Assuming contrarily that $n \geq 4$, we may choose $\overline{v}_j$ and $\overline{v}_k$ which are subspaces of $P$ and are not equal to $\overline{v}_\ell$. In particular $Bv_j, Bv_k \in P$ since only $v_\ell$ is mapped outside $P$ by $B$. Since $P$ is two-dimensional we may write $v_\ell=\alpha v_j+\beta v_k$ for some $\alpha,\beta \neq 0$. We have $\overline{Bv}_\ell=\overline{v}_1=\overline{u}$ which is transverse to $P$, but since $\overline{Bv}_j,\overline{Bv}_k$ are subspaces of $P$ we have $Bv_\ell=\alpha Bv_j+\beta Bv_k \in P$ which is a contradiction. We conclude that $n=3$ and the set $\{\overline{v}_1,\overline{v}_2,\overline{v}_3\}$ is preserved by $G$. Since $\overline{v}_1$ is transverse to $P$, and $\overline{v}_2$, $\overline{v}_3$ belong to $P$ and are distinct from one another, $\{\overline{v}_1,\overline{v}_2,\overline{v}_3\}$ is linearly independent. The result follows.
\end{proof}

The case of Theorem \ref{thm:3d} in which $\mathsf{A}$ is irreducible but not $\varphi^s$-quasimultiplicative now follows by combining Proposition \ref{thm:3dim-permutation} with Theorem \ref{thm:intro-perm}.

\subsection{The reducible case}

If $\mathsf{A}$ is simultaneously upper triangularisable, then, by Theorem \ref{thm:intro-block}, the $\varphi^s$-equilibrium states of $\mathsf{A}$ are the same as those of the corresponding diagonal matrices. By Theorem \ref{thm:intro-perm}, these equilibrium states are fully supported and  at most six ergodic equilibrium states exist. The remaining reducible cases may be reduced as follows.

\begin{proposition}\label{pr:3d-reducible}
  Let $0<s<3$ and $\mathsf{A}=(A_1,\ldots,A_N) \in GL_3(\mathbb{R})^N$ be reducible but not simultaneously upper triangularisable. Then there exist $(b_1,\ldots,b_N)\in (\mathbb{R} \setminus \{0\})^N$ and irreducible $\mathsf{C}=(C_1,\ldots,C_N)\in GL_2(\mathbb{R})^N$ such that the $\varphi^s$-equilibrium states of $\mathsf{A}$ are precisely the $\varphi^s$-equilibrium states of the tuple $\mathsf{A}'=(A_1',\ldots,A_N')$ in which
  \begin{equation*}
    A_i' =
    \begin{pmatrix}
      b_i & 0 \\
      0   & C_i
    \end{pmatrix}
  \end{equation*}
  for all $i \in \{ 1,\ldots,N \}$.
\end{proposition}

\begin{proof}
  If $\mathsf{A}$ preserves a $1$-dimensional subspace of $\mathbb{R}^3$, then there exist $(b_1,\ldots,b_N)\in (\mathbb{R} \setminus \{0\})^N$, $\mathsf{C}=(C_1,\ldots,C_N)\in GL_2(\mathbb{R})^N$, $1\times2$ matrices $D_1,\ldots,D_N$, and a change of basis matrix $X$ such that
  \[
    X^{-1}A_iX =
    \begin{pmatrix}
      b_i & D_i \\
      0   & C_i
    \end{pmatrix}
  \]
  for all $i \in \{ 1,\ldots,N \}$. If $\mathsf{C}$ is reducible, then by a further change of basis we see that $\mathsf{A}$ is simultaneously upper triangularisable which is a contradiction, so $\mathsf{C}$ must be irreducible. By Theorem \ref{thm:intro-block}, the set of equilibrium states is unchanged if we replace the matrices $D_i$ with zero. This completes the proof in the case where $\mathsf{A}$ preserves a $1$-dimensional subspace.

  If $\mathsf{A}$ preserves a $2$-dimensional subspace of $\mathbb{R}^3$, then there instead exist $(b_1,\ldots,b_N)\in (\mathbb{R} \setminus \{0\})^N$, $\mathsf{C}=(C_1,\ldots,C_N)\in GL_2(\mathbb{R})^N$, $2\times1$ matrices $D_1,\ldots,D_N$, and a change of basis matrix $X$ such that
  \[
    X^{-1}A_iX =
    \begin{pmatrix}
      C_i & D_i \\
      0   & b_i
    \end{pmatrix}
  \]
  for all $i \in \{ 1,\ldots,N \}$. If $\mathsf{C}$ is reducible then $\mathsf{A}$ is upper triangularisable which is a contradiction, so we again find that $\mathsf{C}$ is irreducible. Again, by Theorem \ref{thm:intro-block}, the set of equilibrium states is unchanged if we replace the matrices $D_i$ with zero. Since 
  \[
    \begin{pmatrix}
      0 & 1 & 0 \\
      0 & 0 & 1 \\
      1 & 0 & 0
    \end{pmatrix}^{-1}
    \begin{pmatrix}
      C_i & 0 \\
      0 & b_i
    \end{pmatrix}
    \begin{pmatrix}
      0 & 1 & 0 \\
      0 & 0 & 1 \\
      1 & 0 & 0
    \end{pmatrix}
    =
    \begin{pmatrix}
      b_i & 0 \\
      0 & C_i
    \end{pmatrix}
  \]
  for all $i \in \{ 1,\ldots,N \}$ we have finished the proof.
\end{proof}

The following result now suffices to complete the proof of Theorem \ref{thm:3d}.

\begin{proposition}\label{pr:3d-final}
  Let $1<s<2$ and $\mathsf{A} = (A_1,\ldots,A_N) \in GL_3(\R)^N$ be such that
  \begin{equation*}
    A_i =
    \begin{pmatrix}
      b_i & 0 \\
      0 & C_i
    \end{pmatrix}
  \end{equation*}
  for all $i \in \{ 1,\ldots,N \}$, where $\mathsf{B} = (b_1,\ldots,b_N) \in (\R \setminus \{0\})^N$ and $\mathsf{C} = (C_1,\ldots,C_N) \in GL_2(\R)^N$ is irreducible. Then there exist at most three distinct ergodic $\varphi^s$-equilibrium states of $\mathsf{A}$ and they are all fully supported.
\end{proposition}

\begin{proof}
  If $\nu \in \MM_\sigma(\Sigma)$ is ergodic then it is easily seen that the three Lyapunov exponents of $\nu$ are $\lambda_{\mathsf{C}}(\alpha_1,\nu)$, $\lambda_{\mathsf{C}}(\alpha_2,\nu)$, and $\lambda_{\mathsf{B}}(\alpha_1,\nu)$ in some order, with $\lambda_{\mathsf{C}}(\alpha_2,\nu)$ not preceding $\lambda_{\mathsf{C}}(\alpha_1,\nu)$. Thus there are three possibilities:
  \begin{itemize}
    \item[(1)] $\lambda_{\mathsf{B}}(\alpha_1,\nu) \ge \lambda_{\mathsf{C}}(\alpha_1,\nu) \ge \lambda_{\mathsf{C}}(\alpha_2,\nu)$,
    \item[(2)] $\lambda_{\mathsf{C}}(\alpha_1,\nu) \ge \lambda_{\mathsf{B}}(\alpha_1,\nu) \ge \lambda_{\mathsf{C}}(\alpha_2,\nu)$,
    \item[(3)] $\lambda_{\mathsf{C}}(\alpha_1,\nu) \ge \lambda_{\mathsf{C}}(\alpha_2,\nu) \ge \lambda_{\mathsf{B}}(\alpha_1,\nu)$.
  \end{itemize}
  Let $\mu \in \MM_\sigma(\Sigma)$ be an ergodic $\fii^s$-equilibrium state of $\mathsf{A}$. By the definition, it satisfies
  \begin{equation*}
    h(\mu) + \lambda_{\mathsf{A}}(\fii^s,\mu) = \sup\{ h(\nu) + \lambda_{\mathsf{A}}(\fii^s,\nu) : \nu \in \MM_\sigma(\Sigma) \}.
  \end{equation*}
  Observe that, since $1<s<2$, we respectively have three possibilities:
  \begin{itemize}
    \item[(1)] $\lambda_{\mathsf{A}}(\fii^s,\nu) = \lambda_{\mathsf{B}}(\alpha_1,\nu) + (s-1)\lambda_{\mathsf{C}}(\alpha_1,\nu)$,
    \item[(2)] $\lambda_{\mathsf{A}}(\fii^s,\nu) = \lambda_{\mathsf{C}}(\alpha_1,\nu) + (s-1)\lambda_{\mathsf{B}}(\alpha_1,\nu)$,
    \item[(3)] $\lambda_{\mathsf{A}}(\fii^s,\nu) = \lambda_{\mathsf{C}}(\alpha_1,\nu) + (s-1)\lambda_{\mathsf{C}}(\alpha_2,\nu)$,
  \end{itemize}
  for all $\nu \in \MM_\sigma(\Sigma)$. We treat these three cases separately and show that, in each case, $\mu$ is a $\|\cdot\|^t$-equilibrium state of an auxiliary irreducible matrix tuple for some $t>0$. By Theorem \ref{thm:FeKa}, this auxiliary tuple has exactly one $\|\cdot\|^t$-equilibrium state and this equilibrium state is fully supported. It follows that at most three possible candidates exist for the ergodic $\fii^s$-equilibrium state $\mu$ of $\mathsf{A}$, and all three are fully supported.
  
  In the first case, we choose the auxiliary irreducible tuple of matrices $\mathsf{A}' = (A_1',\ldots,A_N') \in GL_2(\R)^N$ such that $A_i' = |b_i|^{1/(s-1)} C_i$ for all $i \in \{ 1,\ldots,N \}$. Since
  \begin{equation*}
    \lambda_{\mathsf{A}}(\fii^s,\nu) = \lambda_{\mathsf{B}}(\alpha_1,\nu) + (s-1)\lambda_{\mathsf{C}}(\alpha_1,\nu) = (s-1)\lambda_{\mathsf{A}'}(\alpha_1,\nu)
  \end{equation*}
  for all $\nu \in \MM_\sigma(\Sigma)$ we see that $\mu$ is a $\|\cdot\|^{s-1}$-equilibrium state of $\mathsf{A}'$. In the second case, we let $\mathsf{A}'' = (A_1'',\ldots,A_N'') \in GL_2(\R)^N$ be such that $A_i'' = |b_i|^{s-1}C_i$ for all $i \in \{ 1,\ldots,N \}$. We note that
  \begin{equation*}
    \lambda_{\mathsf{A}}(\fii^s,\nu) = \lambda_{\mathsf{C}}(\alpha_1,\nu) + (s-1)\lambda_{\mathsf{B}}(\alpha_1,\nu) = \lambda_{\mathsf{A}''}(\alpha_1,\nu)
  \end{equation*}
  for all $\nu \in \MM_\sigma(\Sigma)$ and therefore $\mu$ is a $\|\cdot\|$-equilibrium state of the irreducible matrix tuple $\mathsf{A}''$. In the third case, we define $\mathsf{A}''' = (A_1''',\ldots,A_N''') \in GL_2(\R)^N$ such that $A_i''' = |\det(C_i)|^{(2-s)/(s-1)}C_i$ for all $i \in \{ 1,\ldots,N \}$. Since
  \begin{equation*}
    \lambda_{\mathsf{A}}(\fii^s,\nu) = \lambda_{\mathsf{C}}(\alpha_1,\nu) + (s-1)\lambda_{\mathsf{C}}(\alpha_2,\nu) = (s-1)\lambda_{\mathsf{A}'''}(\alpha_1,\nu)
  \end{equation*}
  for all $\nu \in \MM_\sigma(\Sigma)$ we conclude that $\mu$ is a $\|\cdot\|^{s-1}$-equilibrium state of the irreducible matrix tuple $\mathsf{A}'''$.
\end{proof}

\subsection{Remarks on higher-dimensional cases} \label{sec:remarks-high}
It is instructive to count the ways in which these arguments are inadequate for the problem of understanding $\varphi^s$-equilibrium states in four dimensions. Firstly in four dimensions there exist cases where $\mathsf{A}$ is strongly irreducible but $\mathsf{A}^{\wedge 2}$ is reducible (this can occur for example if $\mathsf{A}\in SO(4)^N$) and therefore Theorem \ref{thm:intro-alg} cannot be applied, so additional arguments are required in order to understand the strongly irreducible case. Secondly if $\mathsf{A}$ is irreducible but not strongly irreducible it may fail to be the case that $\mathsf{A}$ preserves a finite union of $1$-dimensional or $1$-codimensional subspaces. Thus the reduction to a generalised permutation matrix is impossible, and additional arguments are required in this case too. Thirdly, in the reducible case one encounters $\mathsf{A} = (A_1,\ldots,A_N)$ of the form
\[
  A_i =
  \begin{pmatrix}
    B_i & 0 \\
    0 & C_i
  \end{pmatrix}
\]
where $(B_1,\ldots,B_N) \in GL_2(\mathbb{R})^N$ and $(C_1,\ldots,C_N) \in GL_2(\mathbb{R})^N$ are both irreducible. Currently we do not know of any mechanism for resolving this case.


\section{Affinity dimension} \label{sec:affinity}

Knowing that $\fii^s$-equilibrium states are fully supported, we may further study the properties of the singular value pressure. We will observe that, as a consequence of Theorem \ref{thm:3d}, removing one matrix from the tuple $\mathsf{A}$ causes a strict drop in the value of the singular value pressure at every $s \in [0,d]$. This, together with Falconer \cite[Theorem 5.3]{Falconer1988}, will then imply Theorem \ref{thm:schmeling}. We remark that the two-dimensional version of this result is known; it follows from Theorem \ref{thm:FeKa2}. Also, the result holds in any class of self-affine sets where the dimension is obtained from the affinity dimension.

\begin{proposition} \label{thm:pressure-drop}
  Let $0 \leq  s \le d$ and $\mathsf{A} = (A_1,\ldots,A_N) \in GL_d(\R)^N$. If all the $\fii^s$-equilibrium states of $\mathsf{A}$ are fully supported and $\mathsf{A}' = (A_1,\ldots,A_{N-1}) \in GL_d(\R)^{N-1}$, then $P_{\mathsf{A}'}(\fii^s) < P_{\mathsf{A}}(\fii^s)$. Moreover, if $P_{\mathsf{A}}(\fii^d) \le 0$, then $\dimaff(\mathsf{A}') < \dimaff(\mathsf{A})$.
\end{proposition}

\begin{proof}
  Let $\mu \in \MM_\sigma(\Sigma)$ be a $\fii^s$-equilibrium state of $\mathsf{A}$. Then
  \begin{equation} \label{eq:pressure-drop}
    h(\nu) + \lambda_{\mathsf{A}}(\fii^s,\nu) < h(\mu) + \lambda_{\mathsf{A}}(\fii^s,\mu) = P_{\mathsf{A}}(\fii^s)
  \end{equation}
  for all $\nu \in \MM_\sigma(\Sigma)$ which are not $\fii^s$-equilibrium states of $\mathsf{A}$. If $\nu \in \MM_\sigma(\Sigma)$ is a $\fii^s$-equilibrium state of $\mathsf{A}'$, then it is supported on $\{ 1,\ldots,N-1 \}^{\N} \subsetneq \Sigma$ and thus, by the assumption, it cannot be a $\fii^s$-equilibrium state of $\mathsf{A}$. Therefore, \eqref{eq:pressure-drop} gives the first claim. The second claim follows immediately from this since the singular value pressure, as a function of $s$, is continuous and strictly decreasing.
\end{proof}

\begin{remark}
  If $\mathsf{A} = (A_1,\ldots,A_N) \in GL_d(\R)^N$, all the $\fii^s$-equilibrium states of $\mathsf{A}$ are fully supported, and $\Gamma$ is a proper nonempty closed subset of $\Sigma$ satisfying $\sigma(\Gamma) \subset \Gamma$, then we may similarly show that
  \begin{equation*}
    \lim_{n \to \infty} \tfrac{1}{n} \log\sum_{\iii \in \Gamma_n} \fii^s(A_\iii) < P_{\mathsf{A}}(\fii^s).
  \end{equation*}
  Note that the limit above exists since $\iii \in \Gamma_n$ and $\jjj \in \Gamma_m$ whenever $\iii\jjj \in \Gamma_{n+m}$; see \cite[\S 2]{KaenmakiVilppolainen2010}. Here $\Gamma_n = \{ \iii|_n \in \Sigma_n : \iii \in \Gamma \}$. Indeed, there exists $\iii \in \Sigma_*$ which does not appear in any element of $\Gamma$; see \cite[\S 2.1]{KaenmakiRossi2016}. Since $\Gamma \subset \{ \jjj_1\jjj_2\cdots \in \Sigma : |\jjj_k|=|\iii| \text{ and } \jjj_k \ne \iii \text{ for all } k \in \N \}$ and, by iterating, we may assume that $|\iii|=1$ the claim follows from Proposition \ref{thm:pressure-drop}.
\end{remark}


\section{Examples}

In the final section, we present couple of examples. In Example \ref{ex:irred}, we demonstrate that the irreducible but not quasimultiplicative case considered in the course of the proof of Theorem \ref{thm:3d} in \S \ref{sec:irred-case} is nonempty. In Example \ref{ex:no-schmeling}, we exhibit degenerate self-affine sets for which Theorem \ref{thm:schmeling} does not hold.

\begin{lemma}\label{le:irred}
  Let $\mathsf{A} = (A_1,A_2) \in GL_3(\mathbb{R})^2$ be such that
  \[
    A_1 =
    \begin{pmatrix}
      0 & a & 0 \\
      0 & 0 & b \\
      c & 0 & 0
    \end{pmatrix}
    \quad \text{and} \quad A_2 = A_1^T =
    \begin{pmatrix}
      0 & 0 & c \\
      a & 0 & 0 \\
      0 & b & 0
    \end{pmatrix},
  \]
  where $a,b,c \in \mathbb{R} \setminus \{ 0 \}$ and at least two of the numbers $a^2$, $b^2$, $c^2$ are distinct. Then $\mathsf{A}$ is irreducible.
\end{lemma}

\begin{proof}
  Let us suppose for a contradiction that $\mathsf{A}$ is reducible. Clearly the common invariant subspace has dimension either $1$ or $2$. We may assume that it has dimension $1$. Indeed, if it has dimension $2$, then its orthogonal complement has dimension $1$ and is preserved by $A_1^T$ and $A_2^T$, i.e.\ $A_2$ and $A_1$. This $1$-dimensional space must be an eigenspace of $A_1$ and of $A_2$ and therefore the two matrices commute on it, so $A_1A_2-A_2A_1$ maps this subspace to zero.

  To complete the proof we will show that $\mathrm{ker}(A_1A_2-A_2A_1)$ cannot contain a subspace which is invariant under either $A_1$ or $A_2$. Since
  \[
    A_1A_2=
    \begin{pmatrix}
      a^2 & 0 & 0 \\
      0 & b^2 & 0 \\
      0 & 0 & c^2
    \end{pmatrix}
    \quad \text{and} \quad A_2A_1=
    \begin{pmatrix}
      c^2 & 0 & 0 \\
      0 & a^2 & 0 \\
      0 & 0 & b^2
    \end{pmatrix}
  \]
  we have
  \[
    A_1A_2-A_2A_1=
    \begin{pmatrix}
      a^2-c^2 & 0 & 0 \\
      0 & b^2-a^2 & 0 \\
      0 & 0 & c^2-b^2
    \end{pmatrix}.
  \]
  If $a^2$, $b^2$, $c^2$ are all distinct, then the determinant of this matrix is nonzero and its kernel is simply $\{0\}$. Otherwise the numbers $a^2$, $b^2$, $c^2$ take exactly two distinct values and so exactly one of the diagonal entries is zero. This implies that $\mathrm{ker}(A_1A_2-A_2A_1)$ is one of the three coordinate axes. Since there is no coordinate axis that is preserved by either $A_1$ or $A_2$ we have achieved a contradiction.
\end{proof}

\begin{example} \label{ex:irred}
  In this example, we exhibit an irreducible tuple $\mathsf{A}$ of $3 \times 3$ matrices for which the $\fii^s$-equilibrium state is not unique. Let $d=3$, $k=1<s<2$, and note that $(d-k)\binom{d}{k}=6$. Choose $\lambda \in \mathbb{R}$ such that $|\lambda|\notin \{0,1\}$ and define
  \[
    A_1 =
    \begin{pmatrix}
      0 & 0 & \lambda \\
      1 & 0 & 0 \\
      0 & \lambda & 0
    \end{pmatrix},
    \quad \text{and} \quad A_2 =
    \begin{pmatrix}
      0 & 1 & 0 \\
      0 & 0 & \lambda \\
      \lambda & 0 & 0
    \end{pmatrix}.
  \]
  By Lemma \ref{le:irred}, we see that $\mathsf{A} = (A_1,A_2) \in GL_3(\R)^2$ is irreducible. The basis for $\mathbb{R}^6$ used in Proposition \ref{thm:ddim-permutation} is then given by $e_{\{1\},2}, e_{\{2\},3},e_{\{3\},1},e_{\{1\},3}, e_{\{2\},1},e_{\{3\},2}$,
  and in this basis we have
  \[
    \mathfrak{h}_s(A_1) =
    \begin{pmatrix}
      0 & 0 & |\lambda| & 0 & 0 & 0 \\
      |\lambda|^{s-1} & 0 & 0 & 0 & 0 & 0 \\
      0 & |\lambda|^s & 0 & 0 & 0 & 0 \\
      0 & 0 & 0 & 0 & 0 & |\lambda|^s \\
      0 & 0 & 0 & |\lambda|^{s-1} & 0 & 0 \\
      0 & 0 & 0 & 0 & |\lambda| & 0
    \end{pmatrix}
  \]
  and
  \[
    \mathfrak{h}_s(A_2) =
    \begin{pmatrix}
      0 & |\lambda|^{s-1} & 0 & 0 & 0 & 0 \\
      0 & 0 & |\lambda|^s & 0 & 0 & 0 \\
      |\lambda| & 0 & 0 & 0 & 0 & 0 \\
      0 & 0 & 0 & 0 & |\lambda|^{s-1} & 0 \\
      0 & 0 & 0 & 0 & 0 & |\lambda| \\
      0 & 0 & 0 & |\lambda|^s & 0 & 0
    \end{pmatrix},
  \]
  where $\mathfrak{h}_s \colon P_3(\R) \to P_6(\R)$ is as in Proposition \ref{thm:ddim-permutation}. Thus, if we write
  \begin{alignat*}{2}
    B_1 &=
    \begin{pmatrix}
      0 & 0 & |\lambda| \\
      |\lambda|^{s-1} & 0 & 0 \\
      0 & |\lambda|^s & 0
    \end{pmatrix},
    &\qquad B_2 &=
    \begin{pmatrix}
      0 & |\lambda|^{s-1} & 0 \\
      0 & 0 & |\lambda|^s \\
      |\lambda| & 0 & 0
    \end{pmatrix}, \\
    D_1 &=
    \begin{pmatrix}
      0 & 0 & |\lambda|^s \\
      |\lambda|^{s-1} & 0 & 0 \\
      0 & |\lambda| & 0
    \end{pmatrix},
    & D_2 &=
    \begin{pmatrix}
      0 & |\lambda|^{s-1} & 0 \\
      0 & 0 & |\lambda| \\
      |\lambda|^s & 0 & 0
    \end{pmatrix},
  \end{alignat*}
  then
  \[
    \mathfrak{h}_s(A_1) =
    \begin{pmatrix}
      B_1 & 0 \\
      0 & D_1
    \end{pmatrix}
    \quad \text{and} \quad \mathfrak{h}_s(A_2) =
    \begin{pmatrix}
      B_2 & 0 \\
      0 & D_2
    \end{pmatrix}.
  \]
  By Lemma \ref{le:irred}, both $\mathsf{B} = (B_1,B_2) \in GL_3(\R)^2$ and $\mathsf{D} = (D_1,D_2) \in GL_3(\R)^2$ are irreducible. Defining
  \[
    X = X^{-1} =
    \begin{pmatrix}
      0 & 1 & 0 \\
      1 & 0 & 0 \\
      0 & 0 & 1
    \end{pmatrix},
  \]
  it is easy to see that
  \[
    XB_1X^{-1} = D_2 \quad \text{and} \quad XB_2X^{-1} = D_1.
  \]
  Indeed, it suffices only to check one of these two equations directly, since the relations $B_1^T=B_2$, $D_1^T=D_2$, and $X^T=X^{-1}=X$ imply that taking the transpose of either equation transforms it into the other. In particular, it follows that $\mathsf{B}$ and $\mathsf{D}$ have the same norm pressure, $P_{\mathsf{B}}(\|\cdot\|)=P_{\mathsf{D}}(\|\cdot\|)$.

  By Corollary \ref{co:permutation-equilibrium}, the $\varphi^s$-equilibrium states of $\mathsf{A}$ are precisely the $\|\cdot\|$-equilibrium states of $(\mathfrak{h}_s(A_1),\mathfrak{h}_s(A_2))$, and by Theorem \ref{thm:FeKa}, or more precisely, by \cite[Theorem 1.7]{FengKaenmaki2011}, the ergodic $\|\cdot\|$-equilibrium states of that pair of matrices are precisely the ergodic $\|\cdot\|$-equilibrium states of the two pairs $\mathsf{B}$ and $\mathsf{D}$. By irreducibility, it follows that $\mathsf{B}$ and $\mathsf{D}$ have exactly one $\|\cdot\|$-equilibrium state each, which we denote by $\mu_{\mathsf{B}}$ and $\mu_{\mathsf{D}}$, respectively, and
  \[h(\mu_{\mathsf{B}})+\lambda_{\mathsf{B}}(\|\cdot\|,\mu)=P_{\mathsf{B}}(\|\cdot\|)=P_{\mathsf{D}}(\|\cdot\|)=h(\mu_{\mathsf{D}})+\lambda_{\mathsf{D}}(\|\cdot\|,\mu).\]
  To show that $\mathsf{A}$ has exactly $2$ ergodic $\varphi^s$-equilibrium states it is therefore necessary and sufficient to show that $\mu_{\mathsf{B}} \neq \mu_{\mathsf{D}}$.

  It was also shown in \cite[Theorem 1.7]{FengKaenmaki2011} that the measures $\mu_{\mathsf{B}}$ and $\mu_{\mathsf{D}}$ satisfy the following Gibbs property: there exists $C \ge 1$ such that
  \[
    C^{-1}e^{-|\iii|P_{\mathsf{B}}(\|\cdot\|)}\|B_\iii\| \le \mu_{\mathsf{B}}([\iii]) \le Ce^{-|\iii|P_{\mathsf{B}}(\|\cdot\|)}\|B_\iii\|
  \]
  and
  \[
    C^{-1}e^{-|\iii|P_{\mathsf{D}}(\|\cdot\|)}\|D_\iii\| \le \mu_{\mathsf{D}}([\iii]) \le Ce^{-|\iii|P_{\mathsf{D}}(\|\cdot\|)}\|D_\iii\|
  \]
  for all $\iii \in \Sigma_*$. In particular, by the spectral radius formula, this implies
  \[
    \lim_{n\to\infty} \mu_{\mathsf{B}}([\iii^n])^{1/n} = e^{-|\iii|P_{\mathsf{B}}(\|\cdot\|)} \lim_{n\to\infty} \|B_{\iii}^n\|^{1/n} = e^{-|\iii|P_{\mathsf{B}}(\|\cdot\|)}\rho(B_{\iii})
  \]
  and
  \[
    \lim_{n\to\infty} \mu_{\mathsf{D}}([\iii^n])^{1/n} = e^{-|\iii|P_{\mathsf{D}}(\|\cdot\|)}\lim_{n\to\infty} \|D_{\iii}^n\|^{1/n} = e^{-|\iii|P_{\mathsf{D}}(\|\cdot\|)}\rho(D_{\iii})=e^{-|\iii|P_{\mathsf{B}}(\|\cdot\|)}\rho(D_{\iii})
  \]
  for all $\iii \in \Sigma_*$. Here $\iii^n$ is the $n$ times concatenation of $\iii$. It follows, in particular, that if $\mu_{\mathsf{B}} = \mu_{\mathsf{D}}$, then $\rho(B_{\iii}) = \rho(D_{\iii})$ for every $\iii \in \Sigma_*$. (In fact, also the converse holds; see \cite{Morris2016}.) To show that $\mu_B \neq \mu_D$ we exhibit a word $\iii$ such that $\rho(B_{\iii})\neq \rho(D_{\iii})$. Since $\rho(AA^T)=\rho(A^TA)=\|A\|^2$ for every $A \in GL_d(\mathbb{R})$ we have
  \begin{align*}
    \rho(B_1^2B_2B_1B_2^2)&=\rho(B_1^2B_2(B_1^2B_2)^T)=\|B_1^2B_2\|^2, \\
    \rho(D_1^2D_2D_1D_2^2)&=\rho(D_1^2D_2(D_1^2D_2)^T)=\|D_1^2D_2\|^2.
  \end{align*}
  So to demonstrate that $\mu_B \neq \mu_D$ it is sufficient to show that $\|B_1^2B_2\|\neq \|D_1^2D_2\|$. Since $1<s<2$ we find that if $|\lambda|>1$ then
  \[
    \|B_1^2B_2\| = \left\|
    \begin{pmatrix}
      0 & 0 & |\lambda|^{2s+1} \\
      |\lambda|^{s+1} & 0 & 0 \\
      0 & |\lambda|^{3s-2} & 0
    \end{pmatrix}
    \right\| = |\lambda|^{2s+1},
  \]
  and
  \[
    \|D_1^2D_2\| = \left\|
    \begin{pmatrix}
      0 & 0 & |\lambda|^{s+2} \\
      |\lambda|^{3s-1} & 0 & 0 \\
      0 & |\lambda|^{2s-1} & 0
    \end{pmatrix}
    \right\| = \max\left\{|\lambda|^{s+2},|\lambda|^{3s-1}\right\}<|\lambda|^{2s+1}
  \]
  so that $\|B_1^2B_2\|<\|D_1^2D_2\|$. If on the other hand $0<|\lambda|<1$ then
  \[\|D_1^2D_2\|=|\lambda|^{2s-1}>\max\left\{|\lambda|^{s+1},|\lambda|^{3s-2}\right\}=\|B_1^2B_2\|.\]
 We conclude that in either case $\rho(B_1^2B_2B_1B_2^2) \neq \rho(D_1^2D_2D_1D_2^2)$, so $\mu_B$ and $\mu_D$ are distinct as claimed and therefore $\mathsf{A}$ has exactly two ergodic $\varphi^s$-equilibrium states.
\end{example}

\begin{example} \label{ex:no-schmeling}
  The most obvious example of an affine IFS that does not satisfy the claim of Theorem \ref{thm:schmeling} is the one where one mapping occurs two times. In this example, we exhibit other degenerate self-affine sets for which the property described in Theorem \ref{thm:schmeling} does not hold. Although, for simplicity, the examples are presented in dimension two, the same phenomenon arises also in dimension three.
  
  (1) Let
  \begin{equation*}
    A =
    \begin{pmatrix}
      \tfrac13 & 0 \\
      0 & \tfrac15
    \end{pmatrix}
    \quad\text{and}\quad
    B =
    \begin{pmatrix}
      \tfrac12 & 0 \\
      0 & \tfrac14
    \end{pmatrix},
  \end{equation*}
  and define $f_i \colon [0,1]^2 \to [0,1]^2$ by setting
  \begin{alignat*}{2}
    f_1(x) &= Ax + (0,\tfrac{4}{10}), &\qquad
    f_4(x) &= Bx + (0,\tfrac14), \\
    f_2(x) &= Ax + (\tfrac13,\tfrac{4}{10}), &
    f_5(x) &= Bx + (\tfrac12,\tfrac14). \\
    f_3(x) &= Ax + (\tfrac23,\tfrac{4}{10}),
  \end{alignat*}
  The self-affine set associated to these five mappings is clearly $[0,1] \times \{ \tfrac12 \}$. It is equally clear that $[0,1] \times \{ \tfrac12 \}$ is the self-affine set associated to any chosen four mappings. Thus there is no dimension drop when one mapping is removed.

  (2) Let
  \begin{equation*}
    A =
    \begin{pmatrix}
      \tfrac13 & 0 \\
      0 & \tfrac14
    \end{pmatrix}
    \quad\text{and}\quad
    B =
    \begin{pmatrix}
      \tfrac13 & 0 \\
      0 & 0
    \end{pmatrix},
  \end{equation*}
  and define $f_i \colon [0,1]^2 \to [0,1]^2$ by setting
  \begin{alignat*}{2}
    f_1(x) &= Ax, &\qquad
    f_4(x) &= Bx + (0,\tfrac{3}{8}), \\
    f_2(x) &= Ax + (\tfrac13,\tfrac34), &
    f_5(x) &= Bx + (0,\tfrac12), \\
    f_3(x) &= Ax + (\tfrac23,0), &
    f_6(x) &= Bx + (0,\tfrac{5}{8}).
  \end{alignat*}
  It should be emphasised that the matrix $B$ is not invertible and therefore this example lies beyond the scope of the results in \S \ref{se:sec2}. Let $F = \bigcup_{i=1}^3 f_i(F)$ and $E = \bigcup_{i=1}^6 f_i(E)$ be the self-affine sets corresponding to the first three and six mappings, respectively. Note that $\dimh(F) = 1$ and $E$ satisfies the strong separation condition. The set $L = \bigcup_{i=4}^6 f_i([0,1]^2)$ is a union of three line segments and hence $\dimh(L)=1$. Since
  \begin{equation*}
    E = L \cup \bigcup_{i=1}^3 f_i(E) = F \cup \bigcup_{\iii \in \bigcup_{n=0}^\infty \{ 1,2,3 \}^n} f_\iii(L)
  \end{equation*}
  we have $\dimh(E) = \max\{ \dimh(F),\dimh(L) \} = 1$. Therefore, removing any of the mappings does not drop the dimension.
\end{example}

\begin{ack}
  The authors thank Pablo Shmerkin for discussions related to the topic of the paper. Ian Morris was supported by the Engineering and Physical Sciences Research Council (grant number EP/L026953/1). In respect of RCUK policies on publicly-funded research data, the authors note that no research data were generated in the course of this research.
\end{ack}


\end{document}